\documentclass[reqno]{amsart}

\input xy
\xyoption{all}
\usepackage{epsfig}
\usepackage{color}
\usepackage{amsthm}
\usepackage{amssymb}
\usepackage{amsmath}
\usepackage{amscd}
\usepackage{amsopn}
\usepackage{url}

\usepackage{color} 

\newcommand{\labell}[1] {\label{#1}}

\numberwithin{equation}{section}
\newtheorem {Theorem}{Theorem}
\numberwithin{Theorem}{section}

\newtheorem {Lemma}[Theorem]    {Lemma}

\newtheorem {Proposition}[Theorem]{Proposition}
\newtheorem {Corollary}[Theorem]{Corollary}
\theoremstyle{definition}
\newtheorem{Definition}[Theorem]{Definition}
\theoremstyle{remark}
\newtheorem{Remark}[Theorem]{Remark}
\newtheorem{Example}[Theorem]{Example}

\expandafter\chardef\csname pre amssym.def at\endcsname=\the\catcode`\@
\catcode`\@=11
\def\undefine#1{\let#1\undefined}
\def\newsymbol#1#2#3#4#5{\let\next@\relax
 \ifnum#2=\@ne\let\next@\msafam@\else
 \ifnum#2=\tw@\let\next@\msbfam@\fi\fi
 \mathchardef#1="#3\next@#4#5}
\def\mathhexbox@#1#2#3{\relax
 \ifmmode\mathpalette{}{\m@th\mathchar"#1#2#3}%
 \else\leavevmode\hbox{$\m@th\mathchar"#1#2#3$}\fi}
\def\hexnumber@#1{\ifcase#1 0\or 1\or 2\or 3\or 4\or 5\or 6\or 7\or 8\or
 9\or A\or B\or C\or D\or E\or F\fi}

\font\teneufm=eufm10
\font\seveneufm=eufm7
\font\fiveeufm=eufm5
\newfam\eufmfam
\textfont\eufmfam=\teneufm
\scriptfont\eufmfam=\seveneufm
\scriptscriptfont\eufmfam=\fiveeufm

\catcode`\@=\csname pre amssym.def at\endcsname

\def    \eps    {\epsilon}

\newcommand{\CA}{{\mathcal A}}

\newcommand{\CI}{{\mathcal I}}

\newcommand{\CM}{{\mathcal M}}

\newcommand{\CS}{{\mathcal S}}

\newcommand{\supp}{\operatorname{supp}}

\newcommand{\id}{{\mathit id}}
\newcommand{\pt}{{\mathit pt}}
\newcommand{\const}{{\mathit const}}

\newcommand{\tF}{\tilde{F}}

\newcommand{\tf}{\tilde{f}}

\newcommand{\teta}{\tilde{\eta}}

\newcommand{\tH}{\tilde{H}}

\newcommand{\A}{{\mathcal A}}
\newcommand{\tA}{\tilde{\mathcal A}}

\newcommand{\PP}{{\mathcal P}}
\newcommand{\bPP}{\bar{\mathcal P}}

\def    \F      {{\mathbb F}}

\def    \C      {{\mathbb C}}
\def    \R      {{\mathbb R}}

\def    \Z      {{\mathbb Z}}

\def    \Q      {{\mathbb Q}}

\def    \T      {{\mathbb T}}
\def    \CP     {{\mathbb C}{\mathbb P}}

\def    \12    {{\frac{1}{2}}}

\def    \p      {\partial}

\def    \ev  {\operatorname{ev}}

\def    \cl     {\operatorname{cl}}
\def    \im     {\operatorname{im}}

\def    \Sp     {\operatorname{Sp}}

\def    \HF     {\operatorname{HF}}
\def    \HQ     {\operatorname{HQ}}
\def    \GW     {\operatorname{GW}}

\def    \CF     {\operatorname{CF}}

\def    \bx     {\bar{x}}
\def    \by     {\bar{y}}
\def    \bz     {\bar{z}}

\def    \MUCZ  {\operatorname{\mu_{\scriptscriptstyle{CZ}}}}
\def    \s  {\operatorname{c}}
\def    \sls  {\operatorname{c}^{\scriptscriptstyle{LS}}}

\def    \ssminus        {\smallsetminus}

\begin{document}


\setlength{\smallskipamount}{6pt}
\setlength{\medskipamount}{10pt}
\setlength{\bigskipamount}{16pt}





\title[Action and Index Spectra and Periodic Orbits]{Action and Index
Spectra and Periodic Orbits in Hamiltonian Dynamics}

\author[Viktor Ginzburg]{Viktor L. Ginzburg}
\author[Ba\c sak G\"urel]{Ba\c sak Z. G\"urel}

\address{VG: Department of Mathematics, UC Santa Cruz,
Santa Cruz, CA 95064, USA}
\email{ginzburg@math.ucsc.edu}

\address{BG: Department of Mathematics, Vanderbilt University,
Nashville, TN 37240, USA} \email{basak.gurel@vanderbilt.edu}

\date{\today} \thanks{The work is partially supported by the NSF and
by the faculty research funds of the University of California, Santa
Cruz.}


\begin{abstract}
The main theme of this paper is the connection between the existence
of infinitely many periodic orbits for a Hamiltonian system and the
behavior of its action or index spectrum under iterations.  We use the
action and index spectra to show that any Hamiltonian diffeomorphism
of a closed, rational manifold with zero first Chern class has
infinitely many periodic orbits and that, for a general rational
manifold, the number of geometrically distinct periodic orbits is
bounded from below by the ratio of the minimal Chern number and half
of the dimension. These generalizations of the Conley conjecture
follow from another result proved here asserting that a Hamiltonian
diffeomorphism with a symplectically degenerate maximum on a closed
rational manifold has infinitely many periodic orbits.

We also show that for a broad class of manifolds and/or Hamiltonian
diffeomorphisms the minimal action--index gap remains bounded for some
infinite sequence of iterations and, as a consequence, whenever a
Hamiltonian diffeomorphism has finitely many periodic orbits, the
actions and mean indices of these orbits must satisfy a certain
relation. Furthermore, for Hamiltonian diffeomorphisms of $\CP^n$
with exactly $n+1$ periodic orbits a stronger result holds. Namely,
for such a Hamiltonian diffeomorphism, the difference of the action
and the mean index on a periodic orbit is independent of the orbit,
provided that the symplectic structure on $\CP^n$
is normalized to be in the same cohomology class as the first Chern class.

\end{abstract}

\maketitle

\tableofcontents

\section{Introduction and main results }
\labell{sec:main-results}

\subsection{Introduction}
\labell{sec:intro}

The main theme of this paper is the interplay between two aspects of
the dynamics of Hamiltonian systems: the existence of periodic orbits
of arbitrarily large period (or just of infinitely many periodic
orbits) and the behavior of the action or (mean) index spectrum under
iterations.

When the manifold is closed and symplectically aspherical, this
interplay is fairly unambiguous and can, for instance, be described as
follows. On the one hand, as is proved in \cite{GG:gap}, the minimal
positive action-index gap remains bounded for a certain sequence of
iterations, and therefore the action-index spectrum ``grows'' with
iteration. On the other hand, every Hamiltonian system on such a
manifold has infinitely many periodic orbits (see Section
\ref{sec:conley}) and this fact can easily be inferred from the
boundedness of the action-index gap; \cite{GG:gap}. In fact, all
symplectic topological proofs of the existence of infinitely many
periodic orbits rely on the analysis of the action or index spectrum;
see, e.g., \cite{Gi:conley,Hi,HZ,SZ,schwarz,Vi:gen}.

However, once the manifold is not assumed to be symplectically
aspherical, the question becomes considerably more involved. The first
reason for this is that a Hamiltonian diffeomorphism of a closed
manifold need not have infinitely many periodic orbits. The second
reason is that, since in general the action and index depend on the
choice of capping of an orbit, the task of extracting information
about the number or growth of the number of orbits from the spectrum
becomes much more difficult.

Here we address the following two questions:

\begin{itemize}

\item Under what conditions on the manifold and/or the action or index
  spectrum, a Hamiltonian diffeomorphism has infinitely many (or just
  many) periodic orbits?

\item What are the special features of the action or index spectrum of
  a Hamiltonian diffeomorphism with only finitely many periodic
  orbits?
\end{itemize}
These two questions, although formally equivalent, represent two very
different, virtually opposite perspectives focusing on mutually
complementary classes of Hamiltonian diffeomorphisms.

The main results of the paper are stated and discussed in detail in
Sections \ref{sec:conley}--\ref{sec:sdm} and the organization of the
paper is outlined in Section \ref{sec:org}.  We refer the reader to
Section \ref{sec:prelim} for necessary definitions and further
references.

\subsection{The Conley conjecture}
\label{sec:conley}

As we understand it today, the Conley conjecture asserts the existence
of infinitely many periodic orbits for any
Hamiltonian diffeomorphism $\varphi$ of a closed, symplectically
aspherical manifold $M$ or, more precisely, the existence of periodic
points of arbitrarily large period, provided that the fixed points are
isolated.  This conjecture was proved for the so-called weakly
non-degenerate Hamiltonian diffeomorphisms in \cite{SZ} and for all
Hamiltonian diffeomorphisms of surfaces, other than $S^2$, in
\cite{FrHa}. In its original form, as stated in \cite{Co} for
$M=\T^{2n}$, the conjecture was established in \cite{Hi} and, finally,
the case of an arbitrary closed, symplectically aspherical manifold
was settled in \cite{Gi:conley}. (See also, e.g.,
\cite{FS,Gu,HZ,schwarz,Vi:gen} for other related results.)  The first
theorem of this paper, proved in Section \ref{sec:conley-pfs}, is an
extension of the Conley conjecture to rational
symplectic manifolds with $c_1(M)\mid_{\pi_2(M)}=0$.

\begin{Theorem}
\labell{thm:conley}
Let $\varphi$ be a Hamiltonian diffeomorphism of a closed, rational
symplectic manifold $(M^{2n},\omega)$ with $c_1(M)\mid_{\pi_2(M)}=0$.
Then $\varphi$ has simple periodic orbits of arbitrarily large period
whenever the fixed points of $\varphi$ are isolated.
\end{Theorem}

The Conley conjecture obviously fails unless the symplectic manifold
(or the diffeomorphism) meets additional requirements such as the
condition $c_1(M)\mid_{\pi_2(M)}=0$ in Theorem \ref{thm:conley}. For
instance, the rotation of $S^2$ in an irrational angle has only two
periodic orbits. (Both of these orbits are fixed points of the
rotation). More generally, we have

\begin{Example}
\label{exam: torus-action}
Let $M$ admit a Hamiltonian torus action with isolated fixed points;
see, e.g., \cite{CdS,GGK,MS:intro} for the definition and further
details. Then a generic element of the torus generates a Hamiltonian
diffeomorphism of $M$ with finitely many periodic orbits and these
orbits are the fixed points of the action. Within this class of
manifolds $M$ are, for instance, the majority of coadjoint orbits of
compact Lie groups. As an explicit example of this type, consider the
Hamiltonian $H=\alpha_0 |z_0|^2+\ldots+\alpha_n |z_n|^2$ on
$\CP^n$. Then $H$ generates a circle action with isolated fixed
points, provided that the eigenvalues $\alpha_0,\ldots,\alpha_n$ are
rationally independent, i.e., linearly independent over
$\Q$. Furthermore, when $M$ is as above, the equivariant blow-up of
$M$ at the fixed points inherits a Hamiltonian torus action and this
action also has, in many instances, isolated fixed points.
\end{Example}

Let $N$ denote the minimal Chern number of $M$; see Section
\ref{sec:conv}. The following result, also established in Section
\ref{sec:conley-pfs}, is of interest when $N$ is large.

\begin{Theorem}
\label{thm:N/n}
Let $M^{2n}$ be a closed, rational, weakly monotone manifold with $2N>
3n$. Then any Hamiltonian diffeomorphism $\varphi$ of $M$ has at least
$\lceil N/n \rceil$ geometrically distinct periodic orbits.
\end{Theorem}

For instance, if $c_1(M)\mid_{\pi_2(M)}=0$, i.e., $N=\infty$, the
theorem asserts the existence of infinitely many geometrically
distinct periodic points, which also follows from Theorem
\ref{thm:conley}.

To put Theorem \ref{thm:N/n} in context, note that as has been
hypothesized by Michael Chance and Dusa McDuff, the Conley conjecture
may hold for closed manifolds with sufficiently many Gromov--Witten
invariants equal to zero. For instance, one might expect the Conley
conjecture to be true, when $N>2n$ (or $N\geq n$ and $M$ is negative
monotone), and hence the quantum product coincides with the cup
product; see Example \ref{ex:quantum=cup}. Theorems \ref{thm:conley}
and \ref{thm:N/n} can be viewed as a step toward proving this
generalization of the Conley conjecture. Note that the manifolds from
Example \ref{exam: torus-action}, for which the Conley conjecture
fails, tend to have a large number of non-zero Gromov--Witten
invariants (see \cite{McD}) and also have $N\leq n+1$ in all known
examples. Among the manifolds satisfying the hypotheses of Theorem
\ref{thm:conley} are Calabi--Yau manifolds. There are also numerous
examples of weakly monotone, rational manifolds with large $N<\infty$,
though none of these manifolds are monotone; see Example
\ref{ex:quantum=cup}.

\subsection{The action--index gap}
\label{sec:gap}
As was pointed out in Section \ref{sec:intro}, all symplectic
geometrical approaches to proving the Conley conjecture type results,
such as Theorems \ref{thm:conley} and \ref{thm:N/n}, rely on the
analysis of the sets of actions or indices of periodic orbits (the
so-called action and mean index spectra).  From this perspective, the
Conley conjecture is quite similar to the degenerate case of the
Arnold conjecture (see Section \ref{sec:LS}) and differs significantly
from the non-degenerate Arnold conjecture whose proof utilizes a
direct count of periodic orbits via Floer homology. Furthermore,
extending the proof of the Conley conjecture beyond the case of a
symplectically aspherical manifold encounters the same difficulty as
the proof of the degenerate Arnold conjecture -- differentiating
between geometrically distinct orbits and recappings of the same
orbit.

The most naive approach to utilizing the action and mean index spectra
in proving the Conley conjecture and related results amounts to
showing that these spectra, modulo the rationality constant
$\lambda_0$ or modulo $2N$, change with iterations.  Arguments of this
type are discussed in more detail in Section \ref{sec:conley-pfs}.  In
this section, following \cite{GG:gap}, we describe a different way of
relating the properties of the action and mean index spectra to the
dynamics of $\varphi$.

To state the main result, we need to introduce some notation. Let $H$
be a one-periodic in time Hamiltonian on $M$. Then, $H$ can also be
viewed as a $k$-periodic Hamiltonian and in this case is denoted by
$H^{(k)}$ and referred to as the $k$th iteration of $H$.  Likewise,
the $k$th iteration of a capped periodic orbit $\bx$ is denoted by
$\bx^k$.  The periodic orbits $\bx$ and $\by$ of $H$ are said to be
geometrically distinct when the periodic orbits through $x(0)$ and
$y(0)$ of the Hamiltonian diffeomorphism $\varphi_H$ generated by $H$
are geometrically distinct.  Let $\CA_{H^{(k)}}(\by)$ and
$\Delta_{H^{(k)}}(\by)$ stand for the action and, respectively, the
mean index of $H$ on a $k$-periodic orbit $\by$.  We refer the reader
to Section \ref{sec:prelim} for a detailed discussion of these
notions. The action gap between two $k$-periodic orbits $\bx$ and
$\by$ of $H$ is then the difference
$\CA_{H^{(k)}}(\bx)-\CA_{H^{(k)}}(\by)$ and the mean index gap is
defined similarly as the difference
$\Delta_{H^{(k)}}(\bx)-\Delta_{H^{(k)}}(\by)$.  Note that with these
definitions, the action or mean index gap can be zero even when $\bx$
and $\by$ are geometrically distinct.  The action--index gap between
$\bx$ and $\by$ is simply the vector in $\R^2$ whose components are
the action and the mean index gaps. Let also $\|H\|$ denote the Hofer
norm of $H$; see Section \ref{sec:prelim}. Finally, recall that an
increasing (infinite) sequence of integers $\nu_1<\nu_2<\ldots$ is
quasi-arithmetic if $\nu_{i+1}-\nu_i$ is bounded from above by a
constant independent of $i$.

\begin{Theorem}[Bounded gap theorem]
\labell{thm:gap1}
  Let $H$ be a Hamiltonian on a closed
  symplectic manifold $(M^{2n},\omega)$ such that all periodic orbits
  of $\varphi_H$ are isolated.  Assume that $(M^{2n},\omega)$ is weakly
  monotone and rational and one of the following conditions holds:

\begin{itemize}

\item[(i)] $\|H\|< \lambda_0$, where $\lambda_0$ is the
  rationality constant of $M$, or
\item[(ii)] $N\geq 2n$.

\end{itemize}
Then there exists a capped one-periodic orbit $\bx$ of $H$, a
quasi-arithmetic sequence of iterations $\nu_i$, and a sequence of capped
$\nu_i$-periodic orbits $\by_i$, geometrically distinct from
$\bx^{\nu_i}$, such that the sequence of action--index gaps
\begin{equation}
\label{eq:gap}
\big(\A_{H^{(\nu_i)}}(\bx^{\nu_i})-\A_{H^{(\nu_i)}}(\by_i),
\Delta_{H^{(\nu_i)}}(\bx^{\nu_i})-\Delta_{H^{(\nu_i)}}(\by_i)\big)
\end{equation}
is bounded.
\end{Theorem}

\begin{Corollary}
\label{cor:gap1}
Let $M$ and $H$ be as in Theorem \ref{thm:gap1}. Then there exists a
quasi-arithmetic sequence of iterations $\nu_i$ and sequences of
geometrically distinct $\nu_i$-periodic orbits $\bz_i$ and $\bz'_i$
such that the sequence of action--index gaps between $\bz_i$ and
$\bz'_i$ is bounded.
\end{Corollary}

\begin{Remark} 
  The rationality assumption on $M$ can be omitted in (ii) when $H$ is
  weakly non-degenerate. Furthermore, regarding the role of the
  condition that all periodic orbits of $H$ are isolated, note that
  the theorem holds trivially without any assumptions on $M$ or $H$
  once this is not the case, i.e., $H$ has infinitely many
  $k$-periodic orbits for some $k$. Observe also that requirement (i)
  is automatically satisfied when $\lambda_0=\infty$, e.g., when $M$
  is symplectically aspherical.
\end{Remark}

Theorem \ref{thm:gap1} is proved in Section \ref{sec:gap-pfs}.  When
$M$ is symplectically aspherical, the theorem (in a slightly stronger
form) was originally established in \cite{GG:gap} where it was also
observed that in this form the theorem implies the Conley conjecture,
cf.\ \cite{HZ,schwarz,Vi:gen}. This is no longer the case when $M$ is
not symplectically aspherical; see Section
\ref{sec:gap-conley}. However, Theorem \ref{thm:gap1} and Theorem
\ref{thm:gap2} stated below do have an application within the realm of
the Conley conjecture, discussed in Section \ref{sec:augm-action}.

Theorem \ref{thm:gap1} concerns the situation when the quantum aspects
of the symplectic topology of $M$ can be neglected. For instance, this
is manifested by the fact that, as has been mentioned above, the
quantum product coincides with the cup product when $N>2n$ or by the
results from \cite{Al,Ke,schwarz2} when $\| H\|<\lambda_0$.

In the other extreme case -- when certain Gromov--Witten invariants of
$M$ are non-zero -- a much stronger version of Theorem \ref{thm:gap1}
holds. Namely, in this case any Hamiltonian $H$ with isolated
one-periodic orbits has a non-zero action--index gap bounded by a
constant independent of the Hamiltonian.  Denote by $\Lambda$ the
Novikov ring of $M$, equipped with the valuation
$I_\omega(A):=-\left<\omega,A\right>$, $A\in\pi_2(M)$, and by $*$ the
product in the quantum homology $\HQ_*(M)$; see Section
\ref{sec:Floer} for the definitions.

\begin{Theorem}[A priori bounded gap theorem]
\labell{thm:gap2}
Let $(M^{2n},\omega)$ be a closed, weakly monotone symplectic
manifold. Assume that there exists $u\in H_{*<2n}(M)$ and $w\in
H_{*<2n}(M)$ and $\alpha\in\Lambda$ such that
\begin{equation}
\label{eq:hom}
[M]=(\alpha u)*w.
\end{equation}
Let $H$ be a Hamiltonian on $M$ with isolated one-periodic
orbits. Assume, in addition, that one of the following conditions
holds:
\begin{itemize}

\item[(a)] $M$ is rational and $I_\omega(\alpha)=\lambda_0$;

\item[(b)] $2n-\deg u<2N$ and $H$ is non-degenerate.
\end{itemize}
Then $H$ has two geometrically distinct capped
one-periodic orbits $\bx$ and $\by$ such that
\begin{equation}
\label{eq:gap21}
0< \A_{H}(\bx)-\A_{H}(\by)<I_\omega(\alpha)
\end{equation}
and
\begin{equation}
\label{eq:gap22}
\big|\Delta_{H}(\bx)-n\big|\leq n
\text{ and }
n\leq\big|\Delta_{H}(\by)-\deg u\big|\leq 2n.
\end{equation}
\end{Theorem}

\begin{Corollary}
\label{cor:gap2}
Let $M$ and $H$ be as in Theorem \ref{thm:gap2}. Then, for every $k$,
the Hamiltonian $H$ has two geometrically distinct $k$-periodic orbits
$\bz$ and $\bz'$ such that the action--index gap between $\bz$ and
$\bz'$ is non-zero and bounded by a constant independent of $H$ and
$k$.
\end{Corollary}

\begin{proof}
The orbits $\bz$ and $\bz'$ of $H^{(k)}$ are, of course, the orbits
$\bx$ and $\by$ from Theorem \ref{thm:gap2} applied to $H^{(k)}$
treated as a one-periodic Hamiltonian. The fact that $\bz$ and $\bz'$
are indeed geometrically distinct readily follows from
\eqref{eq:gap21} and \eqref{eq:gap22}; see the proof of Theorem
\ref{thm:gap2}.
\end{proof}

\begin{Remark}
  The homological condition, \eqref{eq:hom}, imposes a strong restriction
  on the symplectic topology of $M$, even without the additional
  requirement of (a) or (b), and implies, for instance, that $M$
  is strongly uniruled in the sense of \cite{McD}.
\end{Remark}

\begin{Example} 
\label{ex:gap2}
Let us now list some of the manifolds satisfying the hypotheses of
Theorem \ref{thm:gap2}; see, e.g., \cite{MS} and references therein
for relevant calculations of the quantum homology.
\begin{itemize}
\item The complex projective spaces $\CP^n$ and complex Grassmannians
  satisfy both (a) and (b); see also Example \ref{ex:cpn} for
  $\HQ_*(\CP^n)$.

\item Assume that $M$ satisfies (a) or (b) and $P$ is symplectically
  aspherical. Then $M\times P$ satisfies (a) or, respectively, (b).

\item The product $M\times W$ of two rational manifolds satisfies (a)
whenever $M$ does and $\lambda_0(W)=m\lambda_0(M)$, where $m$ is a
positive integer or $\infty$. For instance, (a) holds for the products
$\CP^n\times \CP^{m_1}\times\ldots \CP^{m_r}$ with
$m_1+1,\ldots,m_r+1$ divisible by $n+1$ and equally normalized
symplectic structures.

\item The monotone product $\CP^n\times W$, where $W$ is monotone
  and $\gcd\big(n+1,N(W)\big)\geq 2$, satisfies (b). For instance,
  the monotone product $\CP^{n_1}\times\ldots\times\CP^{n_r}$ meets
  requirement (b) if $\gcd(n_1+1,\ldots,n_r+1)\geq 2$.

\item A not-necessarily monotone product $\CP^1\times\CP^1$ satisfies
  (b); see \cite{Os}.
\end{itemize}
\end{Example}

Theorem \ref{thm:gap2} is proved in Section \ref{sec:LS}. Although the
statement of the theorem is certainly new, the proof follows a
well-familiar path. On the conceptual level, the argument goes back to
the original work \cite{HV} (see also \cite{LT,Lu}), where the main
underlying principle that the presence of a large space of holomorphic
spheres forces the existence of periodic orbits with certain action
bounds is established. Our proof of the theorem relies on the
technique of action selectors, developed in
\cite{HZ,Oh,schwarz,Vi:gen}, and on the Hamiltonian version of the
Ljusternik--Schnirelman theory (rather than on an explicit use of
holomorphic spheres) and bears a resemblance to many an argument found
in, e.g., \cite{Fl,LO,schwarz2} and in \cite{EP1,EP2}. For $M=\CP^n$,
the theorem (with different action and index bounds) also follows from
\cite[Section~3]{EP1}.

\subsection{Augmented action}
\label{sec:augm-action}
Throughout this section, we assume that $M$ is monotone or negative
monotone with monotonicity constant $\lambda<\infty$. Let, as above,
$H$ be a one-periodic in time Hamiltonian on $M$ and let $x$ be a
one-periodic orbit of $H$.  Set
$$
\tA_H(x)=\CA_H(\bx)-\lambda \Delta_H(\bx), 
$$
where $\bx$ is the orbit $x$ equipped with an arbitrary capping. It is
clear that $\tA_H(x)$, referred to as the \emph{augmented action} in
what follows, is independent of the capping. (When $\bx$ is a
$k$-periodic orbit of a $k$-periodic Hamiltonian $H$, the augmented
action is defined in a similar fashion.)  This definition is inspired
by considerations in \cite[Section 1.6]{Sa} and \cite[Section
1.4]{EP2}, where the Conley--Zehnder index is utilized in place of the
mean index. For us, the main advantage of using the mean index is that
$\tA_H(x)$ is defined even when $x$ is degenerate and that the
augmented action is homogeneous, i.e.,
\begin{equation}
\label{eq:tA}
\tA_{H^{(k)}}(x^k)=k\tA_H(x),
\end{equation}
since $\CA_{H^{(k)}}(\bx^k)=k\CA_H(\bx)$ and
$\Delta_{H^{(k)}}(\bx^k)=k\Delta_H(\bx)$.  Furthermore, when $x$ is a
$k$-periodic orbit of a \emph{one}-periodic Hamiltonian $H$, it is
convenient to define the normalized augmented action as
$\tA_H(x)=\tA_{H^{(k)}}(x)/k$; cf.\ Section
\ref{sec:norm-spectr}. Then, by \eqref{eq:tA}, iterating an orbit does
not change the normalized augmented action:
$\tA_H(x^k)=\tA_H(x)$. Moreover, geometrically identical orbits have
equal normalized augmented action.

\begin{Corollary}
\label{cor:augm-action}
Let $M$ be monotone or negative monotone with $\lambda<\infty$ and let
$M$ and $H$ satisfy the hypotheses of Theorem \ref{thm:gap1} or
Theorem \ref{thm:gap2}.  Assume also that $H$ has finitely many
periodic orbits. Then there exist two geometrically distinct
periodic orbits $x$ and $y$ such that
$$
\tA_H(x)=\tA_H(y).
$$
\end{Corollary}

In other words, \emph{$H$ has infinitely many periodic orbits,
provided that any two geometrically distinct orbits have different
normalized augmented actions.} The corollary easily follows from
Theorems \ref{thm:gap1} and \ref{thm:gap2}. Indeed, arguing by
contradiction, assume that $H$ has finitely many periodic orbits and
the difference $\tA_H(x)-\tA_H(y)$ is never zero. Then
$\tA_{H^{(k)}}(x^k)-\tA_{H^{(k)}}(y^k)\to \infty$ as $k\to\infty$ for
any pair of geometrically distinct periodic orbits. As a consequence,
the minimal action--index gap between geometrically distinct
$k$-periodic orbits with arbitrary cappings goes to infinity, which
contradicts Corollary \ref{cor:gap1} or \ref{cor:gap2}.

When $M=\CP^n$, we have the following much more precise result, which
is proved in Section \ref{sec:LS}, along with Theorem \ref{thm:gap2},
as a consequence of the Hamiltonian Ljusternik--Schnirelman theory.

\begin{Theorem}
\labell{thm:CPn} 
Let $H$ be a Hamiltonian on $\CP^n$ with exactly $n+1$ geometrically
distinct periodic orbits $x_0,\ldots,x_n$. Then
$$
\tA_H(x_0)=\ldots =\tA_H(x_n).
$$
\end{Theorem}

Regarding the assumptions of this theorem, recall that every
Hamiltonian on $\CP^n$ has at least $n+1$ one-periodic orbits by the
Arnold conjecture for $\CP^n$ proved in \cite{Fo,FW}; see Section
\ref{sec:LS} for further discussion. Thus, in the setting of Theorem
\ref{thm:CPn}, every periodic orbit $x_i$ is one-periodic.  On the
other hand, hypothetically, every Hamiltonian on $\CP^n$ with more
than $n+1$ one-periodic orbits necessarily has infinitely many
geometrically distinct periodic orbits.  This is a variant of the
Conley conjecture specific to $\CP^n$. (The case of $\CP^1$ is
established in \cite{FrHa}.)  Thus, if we assume this version of the
Conley conjecture, every Hamiltonian on $\CP^n$ with finitely many
periodic orbits must satisfy the requirements of
Theorem~\ref{thm:CPn}.

\begin{Example} 
\label{exam: torus-action2}
Let $M$ and $H$ be as in Example \ref{exam: torus-action} and let $M$
be monotone.  Then $\tA_H(x_i)=\tA_H(x_j)$ for any two one-periodic
orbits of $H$ on $M$. (This fact is not hard to prove using
equivariant cohomology and localization theorems; see, e.g.,
\cite[Appendix C]{GGK}.) For instance, consider the Hamiltonian
$H=c+\sum\alpha_j |z_j|^2$ on $\CP^n$ where the eigenvalues $\alpha_j$
are rationally independent and $c$ is chosen so that $H$ is
normalized: $\int H\omega^n=0$. Then, as a direct calculation shows,
$\tA_H(x_j)=0$ for all $n+1$ fixed points $x_j$.
\end{Example}

This example suggests that an analogue of Theorem \ref{thm:CPn} may
hold for any Hamiltonian $H$ with finitely many periodic orbits on a
monotone manifold or, at least, many (not just two as in Corollary
\ref{cor:augm-action}) of augmented actions $\tA_H(x_i)$ should be
equal. Another question that naturally arises is what $\tA_H(x_i)$ is
equal to in the setting of Theorem \ref{thm:CPn}. Is $\tA_H(x_i)=0$,
when $H$ is normalized, as it is for normalized quadratic Hamiltonians
on $\CP^n$?

\begin{Remark} 

  The augmented action carries nearly as much information as the pair
  $(\Delta_H,\CA_H)$ when $N<\infty$ and $\lambda<\infty$. To
  visualize the relation between the two invariants, let us fix a
  periodic orbit $\bx$. Consider the sequence of points
  $(\Delta_H(\bx\# A),\CA_H(\bx\# A))$, where $A\in \pi_2(M)$, on the
  $(\Delta_H,\CA_H)$-plane.  This sequence lies on the line with slope
  $\lambda$, intersecting the $\CA_H$-axis at $\tA_H(x)$, and the
  adjacent points differ by the vector $(2N,\lambda_0)$. Then, the
  action--index gaps are vectors connecting points from sequences
  corresponding to geometrically distinct periodic orbits of the same
  period. Corollary \ref{cor:augm-action} asserts that after passing
  to a sufficiently high iteration of $H$ at least two sequences lie
  on the same line, when $H$ has finitely many periodic
  orbits. Furthermore, in the setting of Theorem \ref{thm:CPn}, all
  $n+1$ sequences lie on the same line.

\end{Remark}

\subsection{Symplectically non-degenerate maxima}
\label{sec:sdm}
All of the theorems stated above, with the exception of Theorems
\ref{thm:gap2} and \ref{thm:CPn}, rely on the automatic existence of
periodic orbits in the presence of symplectically degenerate
maxima. This phenomenon was essentially discovered in \cite{Hi} and
then further investigated and used in \cite{Gi:conley,GG:gap}.

\begin{Definition}
\label{def:sdm}
An isolated capped periodic orbit $\bx$ of a Hamiltonian $H$ is said
to be a \emph{symplectically degenerate maximum} of $H$ if
$\Delta_H(\bx)=0$ and $\HF_{2n}(H,\bx)\neq 0$.
\end{Definition}

Here, $\HF_*(H,\bx)$ stands for the local Floer homology of $H$ at
$\bx$, graded throughout the paper so that a $C^2$-small
non-degenerate autonomous maximum has degree $2n$.  We refer the
reader to \cite{GG:gap} for a detailed study of symplectically
degenerate maxima and to Section \ref{sec:LFH} for a discussion of the
local Floer homology.  At this point we only note that in Definition
\ref{def:sdm} the condition that $\Delta_H(\bx)=0$ can be replaced by
the requirement that $x$ be totally degenerate and that $H$ and $\bx$
are assumed to have the same period.

\begin{Example}
  Let $\bx$ be an isolated totally degenerate maximum, equipped with
  trivial capping, of an autonomous Hamiltonian. Then $\bx$ is a
  symplectically degenerate maximum.
\end{Example}

The following key result is proved in Section \ref{sec:p-o-sdm} in the
present form and in \cite{Gi:conley} in the case where $M$ is
symplectically aspherical.

\begin{Theorem}
\labell{thm:sdm}
Assume that $(M^{2n},\omega)$ is weakly monotone and rational, and let
$\bx$ be a symplectically degenerate maximum of $H$. Set
$c=\A_H(\bx)$.  Then for every sufficiently small $\eps>0$ there
exists $k_\eps$ such that
$$
\HF^{(kc+\delta_k,\,kc+\eps)}_{2n+1}\big(H^{(k)}\big)\neq 0\text{ for
all $k>k_\eps$ and some $\delta_k$ with $0<\delta_k<\eps$.}
$$
\end{Theorem}

Using Theorem \ref{thm:sdm}, it is not hard to show that $H$ has
infinitely many geometrically distinct periodic orbits whenever it has
a symplectically degenerate maximum. More precisely, in Section
\ref{sec:conley-pfs}, we will prove the following

\begin{Theorem}
\labell{thm:sdm-conley}
Let $H$ be a one-periodic Hamiltonian on a closed, weakly monotone and
rational manifold $M$.

\begin{itemize}

\item[(i)] Assume that some iteration $H^{(k_0)}$ has finitely many
$k_0$-periodic orbits and also has a symplectically non-degenerate
maximum. Then $H$ has infinitely many geometrically distinct periodic
orbits.

\item[(ii)] If, in addition, $k_0=1$ and $\omega\mid_{\pi_2(M)}=0$ or
$c_1(M)\mid_{\pi_2(M)}=0$, the Hamiltonian $H$ has simple periodic
orbits of arbitrarily large period.

\end{itemize}
\end{Theorem}

\begin{Remark}
  Note that the assumption that $M$ is rational plays a purely
  technical role in the proof of Theorem \ref{thm:sdm} and, perhaps,
  can be eliminated entirely. This requirement, however, is also used
  in other ways in the proof of Theorem \ref{thm:N/n} and of Theorem
  \ref{thm:sdm-conley}, except the case where
  $c_1(M)\mid_{\pi_2(M)}=0$. In particular, the rationality
  requirement enters the proof of Theorem \ref{thm:conley} only via
  Theorem \ref{thm:sdm}. We will further discuss these results in
  Section \ref{sec:conley-pfs}.
\end{Remark}

\subsection{Organization of the paper} 
\label{sec:org}
The role of the next section, Section \ref{sec:prelim}, is purely
technical: here we set conventions and notation and recall relevant
results concerning filtered and local Floer homology, quantum
homology, homotopy maps, the mean index, and loops of Hamiltonian
diffeomorphisms. This section is intended for quick reference rather
than for ``linear reading''. The main results of the paper are proved
in the following four sections, which are essentially independent of
each other.  In Section \ref{sec:conley-pfs}, we establish Theorems
\ref{thm:conley}, \ref{thm:N/n}, and \ref{thm:sdm-conley}, along with
some auxiliary results on the action and index spectra. Theorem
\ref{thm:gap1} is proved in Section \ref{sec:gap-pfs}, following
closely \cite{GG:gap}. In Section \ref{sec:p-o-sdm}, we prove Theorem
\ref{thm:sdm}, which the previous results rely on in the case where
the Hamiltonian has a symplectically degenerate maximum. This proof
utilizes a direct sum decomposition of the filtered Floer homology for
small action intervals, which reduces the problem to the case of a
symplectically aspherical manifold (or even $\R^{2n}$), treated in
\cite{GG:duke,GG:gap}, and is of independent interest.  Finally, in
Section \ref{sec:LS}, we discuss the technique of action selectors and
the Hamiltonian Ljusternik--Schnirelman theory and, after a brief
digression into the degenerate case of the Arnold conjecture, prove
Theorems \ref{thm:gap2} and \ref{thm:CPn}.

\subsection{Acknowledgments} The authors are grateful to Dusa McDuff
for useful discussions.

\section{Preliminaries}
\label{sec:prelim}
The goal of this section is to set notation and conventions and to
give a brief review of Floer and quantum homology and several other
notions used in the paper.  Most of this material is quite standard
and the only new results discussed here concern extensions of loops of
Hamiltonian diffeomorphisms; see Section \ref{sec:loops}.

\subsection{Conventions and basic definitions}
\label{sec:conv}

The objective of this subsection is to set terminology and conventions
and recall basic definitions.

\subsubsection{Symplectic manifolds and Hamiltonian flows}
Let $(M,\omega)$ be a symplectic manifold.  Recall that $M$ is said to
be \emph{monotone (negative monotone)} if
$[\omega]\mid_{\pi_2(M)}=\lambda c_1(M)\!\mid_{\pi_2(M)}$ for some
non-negative (respectively, negative) constant $\lambda$ and $M$ is
\emph{rational} if $\left<[\omega], {\pi_2(M)}\right>=\lambda_0\Z$,
i.e., the integrals of $\omega$ over spheres in $M$ form a discrete
subgroup of $\R$. (When $\left<[\omega], {\pi_2(M)}\right>=0$, we set
$\lambda_0=\infty$.)  The constants $\lambda$ and $\lambda_0\geq 0$
will be referred to as the \emph{monotonicity and rationality constants}.
The positive generator $N$ of the discrete subgroup
$\left<c_1(M),\pi_2(M)\right>\subset \R$ is called the \emph{minimal
  Chern number} of $M$. When this subgroup is zero, we set $N=\infty$.
The manifold $M$ is called \emph{symplectically aspherical} if
$c_1(M)\mid_{\pi_2(M)} =0=[\omega]\mid_{\pi_2(M)}$. A symplectically
aspherical manifold is monotone and a monotone or negative monotone
manifold is rational.

To ensure that the standard construction of Floer homology applies,
all symplectic manifolds $(M,\omega)$ are required, throughout the
paper, to be \emph{weakly monotone}, i.e., monotone or 
$N\geq n-2$ (including the case
$c_1(M)\!\mid_{\pi_2(M)}=0$). Utilizing the machinery of virtual
cycles, one can eliminate this requirement in Theorems
\ref{thm:conley}, \ref{thm:N/n}, \ref{thm:gap1}, \ref{thm:sdm}, and
\ref{thm:sdm-conley}.

Furthermore, all Hamiltonians $H$ on $M$ considered in this paper are
assumed to be $k$-periodic in time, i.e., $H\colon S^1_k\times
M\to\R$, where $S^1_k=\R/k\Z$, and the period $k$ is always a positive
integer.  When the period is not specified, it is equal to one, which
is the default period in this paper. We set $H_t = H(t,\cdot)$ for
$t\in S^1=\R/\Z$. The Hamiltonian vector field $X_H$ of $H$ is defined
by $i_{X_H}\omega=-dH$. The (time-dependent) flow of $X_H$ will be
denoted by $\varphi_H^t$ and its time-one map by $\varphi_H$. Such
time-one maps are referred to as \emph{Hamiltonian diffeomorphisms}. A
Hamiltonian $H$ is said to be \emph{normalized} if $\int_M H_t
\omega^n=0$ for all $t\in S^1$. Below, in contrast with, say,
\cite{schwarz}, the Hamiltonians are not by default assumed to be
normalized.

A one-periodic Hamiltonian $H$ can always be treated as
$k$-periodic. In this case, we will use the notation $H^{(k)}$ and,
abusing terminology, call $H^{(k)}$ the $k$th iteration of
$H$. Incorporating the period $k$ into the notation is often redundant
and awkward, but it does eliminate any ambiguity and is convenient
when a Hamiltonian or an orbit can be viewed as $k$-periodic for
different periods $k$.

The Hofer norm of a $k$-periodic Hamiltonian $H$ is defined by
$$
\| H\|=\int_0^k (\max_M H_t-\min_M H_t)\,dt.
$$

Let $K$ and $H$ be two one-periodic Hamiltonians. The composition $K\#
H$ is the Hamiltonian
$$
(K\# H)_t=K_t+H_t\circ(\varphi^t_K)^{-1}
$$
generating the flow $\varphi^t_K\circ \varphi^t_H$. In general,
$K\# H$ is not one-periodic. However, this is the case if, for
example, $H_0\equiv 0\equiv H_1$. The latter condition can be met by
reparametrizing the Hamiltonian as a function of time without changing
the time-one map. Thus, in what follows, we will usually treat $K\# H$
as a one-periodic Hamiltonian.  Another instance when the composition
$K\# H$ of two one-periodic Hamiltonians is automatically one-periodic
is when the flow $\varphi^t_K$ is a loop of Hamiltonian
diffeomorphisms, i.e., $\varphi_K^1=\id$.  We set $H^{\# k}=H\#\ldots
\# H$ ($k$ times). The flow $\varphi^t_{H^{\# k}}=(\varphi_H^t)^k$,
$t\in [0,\,1]$, is homotopic with fixed end-points to the flow
$\varphi^t_H$, $t\in [0,\, k]$. Note also that $\| H^{(k)}\|=k\|H\|=\|H^{\# k}\|$.

\subsubsection{Capped periodic orbits}
\label{sec:capped-per-orb}
Let $x\colon S^1_k\to W$ be a contractible loop. A \emph{capping} of
$x$ is a map $u\colon D^2\to M$ such that $u\mid_{S^1_k}=x$. Two
cappings $u$ and $v$ of $x$ are considered to be equivalent if the
integrals of $\omega$ and $c_1(M)$ over the sphere obtained by
attaching $u$ to $v$ are both equal to zero. For instance, when $M$ is
symplectically aspherical, all cappings of $x$ are equivalent to each
other. A capped closed curve $\bar{x}$ is, by definition, a closed
curve $x$ equipped with an equivalence class of capping. In what
follows, the presence of capping is always indicated by a bar. We
denote by $\PP(H)$ the collection of all one-periodic orbits of $H$
and by $\bPP(H)$ the collection of its capped one-periodic orbits.

The action of a one-periodic Hamiltonian $H$ on a capped closed curve
$\bar{x}=(x,u)$ is defined by
$$
\CA_H(\bar{x})=-\int_u\omega+\int_{S^1} H_t(x(t))\,dt.
$$
The space of capped closed curves is a covering space of the space of
contractible loops and the critical points of $\CA_H$ on the former
space are exactly one-periodic orbits of $X_H$. The \emph{action
spectrum} $\CS(H)$ of $H$ is the set of critical values of
$\CA_H$. This is a zero measure set. When $M$ is rational, $\CS(H)$ is
a closed, and hence nowhere dense, set. Otherwise, $\CS(H)$ is
everywhere dense. Furthermore, when $M$ is rational, the action
$\CA_H(\bx)$, viewed modulo $\lambda_0$, is independent of the
capping.  We denote by $\CS_{\lambda_0}(H)\subset
S^1_{\lambda_0}=\R/\lambda_0\Z$ the action spectrum modulo
$\lambda_0$. These definitions extend to $k$-periodic orbits and
Hamiltonians in an obvious way.

In this paper, we are only concerned with contractible periodic orbits and
\emph{a periodic orbit is always assumed to be contractible, even if
this is not explicitly stated}.

A (capped) periodic orbit $\bar{x}$ of $H$ is \emph{non-degenerate} if
the linearized return map $d\varphi_H \colon T_{x(0)}W\to T_{x(0)}W$
has no eigenvalues equal to one. Following \cite{SZ}, we call $x$
\emph{weakly non-degenerate} if at least one of the eigenvalues is
different from one. Otherwise, the orbit
is said to be \emph{strongly degenerate}. Clearly, capping has no
effect on degeneracy or non-degeneracy of $\bar{x}$. A Hamiltonian is
non-degenerate if all its one-periodic orbits are non-degenerate.

Let $\bar{x}$ be a non-degenerate (capped) periodic orbit.  The
\emph{Conley--Zehnder index} $\MUCZ(\bar{x})\in\Z$ is defined, up to a
sign, as in \cite{Sa,SZ}. More specifically, in this paper, the
Conley--Zehnder index is the negative of that in \cite{Sa}. In other
words, we normalize $\MUCZ$ so that $\MUCZ(\bar{x})=n$ when $x$ is a
non-degenerate maximum (with trivial capping) of an autonomous
Hamiltonian with small Hessian. Sometimes, we will also use the notation
$\MUCZ(H,\bx)$.

As is well-known, the fixed points of $\varphi_H :=\varphi_H^1$ are in
one-to-one correspondence with (not-necessarily contractible)
one-periodic orbits of $H$. Likewise, the $k$-periodic points of
$\varphi_H$, i.e., the fixed points of $\varphi_H^k$, are in
one-to-one correspondence with (not-necessarily contractible)
$k$-periodic orbits of $H$. The $k$th iteration of a one-periodic
orbit $x$ of $H$ is the orbit $x(t)$, where $t$ now ranges in
$[0,\,k]$. It is easy to see that the iteration $\bx^k$ of a capped
orbit $\bx$ carries a natural capping. The action functional is
homogeneous with respect to iteration:
$$
\CA_{H^{(k)}}(\bx^k)=k\CA_H(\bx).
$$
Two periodic orbits $x$ and $y$ with possibly different periods are
said to be \emph{geometrically distinct} if the sets of points
$\varphi_H^i(x(0))=x(i)$ and $\varphi_H^j(y(0))=y(j)$ are distinct,
i.e., if the corresponding periodic orbits of $\varphi_H$ are
geometrically distinct.  Otherwise, we call the orbits
\emph{geometrically identical or equivalent}.  These notions extend to
capped orbits via forgetting the cappings. We will denote the set of
geometrically distinct periodic orbits of $H$ (or, more precisely, the
set of geometrical equivalence classes) by $\PP^\infty(H)$. The set of
such orbits of period less than or equal to $k$ will be denoted by
$\PP_k(H)$. It is worth emphasizing that two $k$-periodic orbits $x$
and $y$ of $H$, with $k>1$, can be distinct (i.e., $x\neq y$) but
geometrically equivalent.  The capped orbits obtained by capping an
orbit $x$ in all possible ways are geometrically identical to each
other and an orbit $x$ and its iteration $x^k$ are geometrically
identical. Two one-periodic orbits $x$ and $y$ are geometrically
distinct if and only if $x(0)\neq y(0)$.  Let $\bx$ and $\by$ be $k$-
and, respectively, $l$-periodic orbits which are geometrically
identical and $M$ is rational. Then
\begin{equation}
\label{eq:g-i1}
\frac{\CA_{H^{(k)}}(\bx)}{k}\equiv \frac{\CA_{H^{(l)}}(\by)}{l}\mod \lambda_0\Q.
\end{equation}
The converse, of course, is not true.

There is a natural action and index preserving one-to-one
correspondence between capped $k$-periodic orbits of $H$ and capped
one-periodic orbits of $H^{\# k}$. Thus, when convenient, we will
treat $\bx^k$ as a one-periodic orbit of $H^{\# k}$.

\subsection{Floer homology and quantum homology}
In this section, we review the construction and basic properties of Floer
homology, mainly to further specify notation and conventions.

Recall that when $M$ is closed and weakly monotone, the filtered Floer
homology of $H\colon S^1\times M\to \R$ for the interval $(a,\, b)$,
denoted throughout the paper by $\HF^{(a,\, b)}_*(H)$, is defined. We
refer the reader to Floer's papers
\cite{F:Morse,F:grad,F:c-l,F:witten} or to, e.g., \cite{HS,MS,Sa} for
further references and introductory accounts of the construction of
(Hamiltonian) Floer homology. The terminology, conventions, and most of
the notation used here are similar to those in
\cite{Gi:coiso,Gi:conley,Gu}. Note however that now
we grade the Floer complex and homology by $\MUCZ+n$. For instance, a
non-degenerate maximum of an autonomous Hamiltonian with small
eigenvalues has degree $2n$. Thus, we have a grading
preserving isomorphism between the total Floer homology $\HF_*(H)$ and
the quantum homology $\HQ_*(M)$.

Fix a ground field $\F$, e.g., $\Z_2$ or $\C$ or $\Q$. Throughout the
paper, $H_*(M)$ denotes the homology $H_*(M;\F)$ with
coefficients in $\F$ and the Floer and quantum homology groups are also
taken over the field $\F$.

\subsubsection{Floer homology}
\label{sec:Floer}
Let $H$ be a non-degenerate Hamiltonian on $M$.
Denote by $\CF^{(-\infty,\, b)}_k(H)$, where $b\in
(-\infty,\,\infty]$ is not in $\CS(H)$, the vector space of formal
sums 
\begin{equation}
\labell{eq:floer-complex}
\sum_{\bar{x}\in \bPP(H)} \alpha_{\bar{x}}\bar{x}. 
\end{equation}
Here $\alpha_{\bar{x}}\in\F$ and $\MUCZ(\bar{x})=k-n$ and
$\CA(\bar{x})<b$. Furthermore, we require, for every $a\in \R$, the
number of terms in this sum with $\alpha_{\bar{x}}\neq 0$ and
$\CA(\bar{x})>a$ to be finite. The graded $\F$-vector space
$\CF^{(-\infty,\, b)}_*(H)$ is endowed with the Floer differential
counting the anti-gradient trajectories of the action functional in
the standard way once a (time-dependent) almost complex structure
compatible with $\omega$ is fixed and the regularity requirements are
satisfied; see, e.g., \cite{HS,MS,Sa}. Thus, we obtain a filtration of
the total Floer complex $\CF_*(H):=\CF^{(-\infty,\,
\infty)}_*(H)$. Furthermore, we set $\CF^{(a,\,
b)}_*(H):=\CF^{(-\infty,\, b)}_*(H)/\CF^{(-\infty,\,a)}_*(H)$, where
$-\infty\leq a<b\leq\infty$ are not in $\CS(H)$. The resulting
homology, the \emph{filtered Floer homology} of $H$, is denoted by
$\HF^{(a,\, b)}_*(H)$ and by $\HF_*(H)$ when
$(a,\,b)=(-\infty,\,\infty)$.

This construction extends to all, not necessarily non-degenerate,
Hamiltonians by continuity. Here, for the sake of simplicity, we
assume that $M$ is rational. Let $H$ be an arbitrary (one-periodic in
time) Hamiltonian on $M$ and let the end points $a$ and $b$ of the
action interval be outside $\CS(H)$. By definition, we set
\begin{equation}
\label{eq:floer-hom}
\HF^{(a,\, b)}_*(H)=\HF^{(a,\, b)}_*(\tH),
\end{equation}
where $\tH$ is a non-degenerate, small perturbation of $H$. It is well
known that the right hand side in \eqref{eq:floer-hom} is independent
of $\tH$ as long as the latter is sufficiently close to $H$. (It is
essential that $a$ and $b$ are not in $\CS(H)$ and $M$ is rational.)
Working with filtered Floer homology, \emph{we will always assume that
the end points of the action interval are not in the action spectrum
and, if $H$ is not non-degenerate, the ambient manifold $M$ is
rational.}

When $a<b<c$, we have $\CF_*^{(b,\,c)}(H)=\CF_*^{(a,\,c)}(H)/\CF_*^{(a,\,b)}(H)$, and thus
obtain the long exact sequence
$$
\ldots\to\HF^{(a,\, b)}_*(H)\to \HF^{(a,\, c)}_*(H)\to \HF^{(b,\, c)}_*(H)\to\ldots\, .
$$

The total Floer complex and homology are modules over the
\emph{Novikov ring} $\Lambda$. In this paper, the latter is defined as
follows. Let $\omega(A)$ and $\left<c_1(M),A\right>$ denote the
integrals of $\omega$ and, respectively, $c_1(M)$ over a cycle $A$. 
Set
$$
I_\omega(A)=-\omega(A)\text{ and }
I_{c_1}(A)=-2\left<c_1(M), A\right>,
$$
where $A\in\pi_2(M)$, and
$$
\Gamma=\frac{\pi_2(M)}{\ker I_\omega\cap \ker I_{c_1}}.
$$
Thus, $\Gamma$ is the quotient of $\pi_2(M)$ by the equivalence
relation where the two spheres $A$ and $A'$ are considered to be
equivalent if $\omega(A)=\omega(A')$ and $\left<c_1(M),
A\right>=\left<c_1(M), A'\right>$. The homomorphisms $I_\omega$ and
$I_{c_1}$ defined originally on $\pi_2(M)$ descend to $\Gamma$.  (With
this convention $I_\omega(A)$ and $I_{c_1}(A)$ have the same sign when
$M$ is monotone and opposite signs when $M$ is negative monotone.)
The group $\Gamma$ acts on $\CF_*(H)$ and on $\HF_*(H)$ via recapping:
an element $A\in \Gamma$ acts on a capped one-periodic orbit $\bar{x}$
of $H$ by attaching the sphere $A$ to the original capping. We denote
the resulting capped orbit by $\bx\# A$. Then,
$$
\deg(\bx\# A)=\deg(\bx)+ I_{c_1}(A)
\text{ and }
\CA_H(\bx\# A)=\CA_H(\bx)+I_\omega(A).
$$
The Novikov ring $\Lambda$ is a certain completion of the group ring
of $\Gamma$. Namely, $\Lambda$ is comprised of formal linear
combinations $\sum \alpha_A e^A$, where $\alpha_A\in\F$ and $A\in
\Gamma$, such that for every $a\in \R$ the sum contains only finitely
many terms with $I_\omega(A) > a$ and, of course, $\alpha_A\neq
0$. (The appearance of $e^A$ rather than just $A$ is due to the fact
that we use addition to denote the product in $\Gamma$ and
multiplication to denote the product in $\Lambda$.) The Novikov ring
$\Lambda$ is graded by setting $\deg(e^A)=I_{c_1}(A)$ for
$A\in\Gamma$.  The action of $\Gamma$ turns $\CF_*(H)$ and $\HF_*(H)$
into $\Lambda$-modules. Note that in general this action is defined
only on the total Floer complex and homology and does not extend to
their filtered counterparts.

\begin{Example}
\label{ex:Novikov}
Assume that $M$ is monotone or negative monotone or rational with
$N=\infty$ (i.e., $c_1(M)\mid_{\pi_2(M)}=0$) and $\lambda_0<\infty$
(i.e., $\omega\mid_{\pi_2(M)}\neq 0$).  Then $\Gamma\cong \Z$ with the
generator $A$ such that $I_\omega(A)=\lambda_0>0$ (i.e.,
$\omega(A)=-\lambda_0$) and $I_{c_1}(A)=\pm 2N$ when $N<\infty$. (The
positive sign for monotone manifolds.) When $N=\infty$, we have
$I_{c_1}(A)=0$.  Furthermore, $\Lambda$ is the ring of Laurent series
$\F[q][[q^{-1}]]$ with $q=e^A$, where $\deg(q)=\pm 2N$ and
$I_\omega(q)=\lambda_0$. When $M$ is symplectically aspherical,
$\Gamma=0$ and $\Lambda\cong\F$.
\end{Example}

Concluding this discussion of Floer homology, two remarks are due.

\begin{Remark}
\label{rmk:floer-gb}
  The complex $\CF_*^{(a,\,b)}(H)$ can be thought of as the vector
  space formed by the formal sums \eqref{eq:floer-complex}, where, in
  addition to other requirements, $a<\CA(\bx)<b$. (Such a sum is
  necessarily finite if $a>-\infty$.) This point of view, taken in,
  e.g., \cite{Gi:coiso,GG:duke,Gu}, enables one to relax the
  compactness requirement on $M$ in the construction of the filtered
  Floer homology.  Namely, let $M$ be weakly monotone, rational, and
  geometrically bounded (see, e.g., \cite{AL,CGK}). Assume furthermore
  that $H$ is compactly supported and $(a,\,b)$ contains no point of
  the subgroup $\left<\omega,\pi_2(M)\right>$. (The latter condition
  forces $M$ to be rational and reduces to $0\not\in (a,\,b)$ when $M$
  is symplectically aspherical.)  Then the above construction of the
  Floer homology goes through word-for-word with the new definition of
  $\CF_*^{(a,\,b)}(H)$ once we notice that in \eqref{eq:floer-hom} it
  suffices to only require the capped periodic orbits of $\tH$ with
  action in $(a,\,b)$ to be non-degenerate and that this requirement
  holds for a generic perturbation $\tH$. This observation is used in
  the proof of Theorem \ref{thm:sdm} in Section \ref{sec:p-o-sdm}.
\end{Remark}

\begin{Remark}
  The condition that $M$ is rational in the definition of filtered
  Floer homology of general Hamiltonians is purely technical and can
  easily be eliminated although \eqref{eq:floer-hom} can no longer be
  used. The reason is that once the rationality requirement is
  dropped, the right hand side of \eqref{eq:floer-hom} depends in
  general on the choice of $\tH$ no matter how close the latter is to
  $H$. Hence, instead of using \eqref{eq:floer-hom}, one can define
  $\HF^{(a,\, b)}_*(H)$ as the limit of $\HF^{(a,\, b)}_*(\tH)$ over a
  certain class of perturbations $\tH$; cf.\ \cite{U1,U2}. For
  instance, we may require that $\tH\leq H$ and $a$ and $b$ are
  outside $\CS(\tH)$. Then the condition that $a$ and $b$ are not in
  $\CS(H)$ is irrelevant and $\HF^{(a,\, b)}_*(H)$ becomes very
  sensitive to variations of $a$ and $b$ and of $H$; cf.\ Section
  \ref{sec:hom}.  Throughout the paper, we mainly restrict our
  attention to the rational case, unless $H$ is assumed to be
  non-degenerate, for the rationality requirement enters the proofs in
  other, more essential, ways.
\end{Remark}

\subsubsection{Homotopy maps}
\label{sec:hom}

A homotopy of Hamiltonians on $M$ is a family of (periodic in time)
Hamiltonians $H^s$ smoothly parametrized by $s\in \R$ and a family of
almost complex structures $J^s$, compatible with $\omega$, such that
$H^s$ and $J^s$ are independent of $s$ for large positive and negative
values of $s$. Throughout this paper, we suppress the family $J^s$ in
the notation and think of a \emph{homotopy} of Hamiltonians on $M$ as
a family of (periodic in time) Hamiltonians $H^s$ smoothly
parametrized by $[0,\,1]$. (Such a family can be easily turned into
one smoothly parametrized by $\R$.)  It is well known that, when $H_0$
and $H_1$ are non-degenerate and certain (generic) regularity
requirements are met, a homotopy gives rise to a chain map of Floer
complexes
$$
\Psi\colon \CF_*(H^0)\to \CF_*(H^1)
$$
that induces an isomorphism on the level of total Floer homology; see,
e.g., \cite{MS} and references therein. The map $\Psi$ does not
preserve the action filtration. However, a standard estimate shows
that $\Psi$ increases the action by no more than
$$
E:=\int_{-\infty}^\infty\int_{S^1} \max_M \p_s H_t^s\,dt\,ds.
$$
As a consequence, for every $c\geq E$, the homotopy induces a map,
which we will also denote by $\Psi$ or $\Psi_{H^0,H^1}$, of the
filtered Floer homology, shifting the action filtration by $c$:
$$
\Psi_{H^0,H^1}\colon \HF_{*}^{(a,\,b)}(H^0)\to \HF_{*}^{(a+c,\,b+c)}(H^1).
$$
This map need not be an isomorphism. For instance, $\Psi=0$ if
$c>b-a$.

\begin{Example}
  Two particular cases of this construction are of interest for us.
  The first case is that of a monotone decreasing homotopy, i.e., a
  homotopy with $\p_s H^s_t\leq 0$. In this case, we can take $c=0$,
  and thus we obtain a map preserving the action filtration. The
  second case is that of a linear homotopy $H^s=sH^1+ (1-s) H^0$. Here
  we can take $c=E=\int \max (H^1-H^0)\, dt$. In what follows,
  whenever a homotopy from $H^0$ to $H^1\geq H^0$ is considered, this
  is always a monotone linear homotopy, unless specified otherwise.
\end{Example}

A (non-monotone) homotopy $H^s$ from $H^0$ to $H^1$ with $a$ and $b$
outside $\CS(H^s)$ for all $s$ gives rise to an isomorphism between
the groups $\HF^{(a,\,b)}_*(H^s)$, and hence, in particular,
\begin{equation}
\labell{eq:viterbo}
\HF^{(a,\,b)}_*(H^0)\cong \HF^{(a,\,b)}_*(H^1);
\end{equation}
see \cite{Vi} and also \cite{BPS,Gi:coiso}.  This isomorphism is
\emph{not} induced by the homotopy in the same sense as
$\Psi_{H^0,H^1}$, but is constructed by breaking the homotopy $H^s$
into a composition of nearly constant homotopies for which
\eqref{eq:viterbo} is essentially immediate.  The isomorphism
\eqref{eq:viterbo} is well-defined, i.e., it is completely determined
by the homotopy and independent of the choices made in its
construction. Furthermore, the isomorphism commutes with the maps from
the long exact sequence, provided that all three points $a<b<c$ are
outside $\CS(H^s)$ for all $s$. When $H^s$ is a decreasing homotopy,
the isomorphism \eqref{eq:viterbo} coincides with $\Psi_{H^0,H^1}$.

\begin{Example}
\labell{ex:isospec} 
A homotopy $H^s$ is said to be
\emph{isospectral} if $\CS(H^s)$ is independent of $s$. In this case,
the isomorphism \eqref{eq:viterbo} is defined for any $a<b$ outside
$\CS(H^s)$.

For instance, let $\eta_s^t$, where $t\in S^1$ and $s\in [0,\, 1]$, be
a family of loops of Hamiltonian diffeomorphisms based at $\id$, i.e.,
$\eta^0_s=\id$ for all $s$. In other words, $\eta_s^t$ is a based
homotopy from the loop $\eta_0^t$ to the loop $\eta_1^t$. Let $G^s_t$
be a family of one-periodic Hamiltonians generating these loops and
let $H$ be a fixed one-periodic Hamiltonian. Then $H^s:=G^s\# H$ is an
isospectral homotopy, provided that $G^s$ are suitably normalized.
(Namely, $\CA(G^s)=0$ for all $s$; see Section \ref{sec:loops}.)
\end{Example}

When $K\geq H^s$ for all $s$,
the isomorphism \eqref{eq:viterbo} intertwines monotone homotopy
homomorphisms from $K$ to $H^0$ and to $H^1$, i.e.,
the diagram
\begin{equation}
\label{eq:diag-iso}
\xymatrix{
{\HF^{(a,\,b)}_*(K)} \ar[d]_{\Psi_{K,H^0}}\ar[rd]^{\Psi_{K,H^1}} &\\
{\HF^{(a,\,b)}_*(H^0)} \ar[r]^{\cong} & {\HF^{(a,\,b)}_*(H^1)}
}
\end{equation}
is commutative. (See, e.g., \cite{Gi:conley} for more details.) The
assumption that $K\geq H^s$ for all $s$ is essential here: the maps
need not commute if we only require that $K\geq H^0$ and $K\geq
H^1$.

\subsubsection{Quantum homology}
The total Floer homology $\HF_*(H)$, equipped with the pair-of-pants
product, is an algebra over the Novikov ring $\Lambda$. This algebra
is isomorphic to the (small) \emph{quantum homology} $\HQ_*(M)$; see,
e.g., \cite{MS,PSS}. On the level of $\Lambda$-modules, we have
\begin{equation}
\label{eq:floer-qh}
\HQ_*(M)=H_*(M)\otimes_\F\Lambda
\end{equation}
with the tensor product grading. Thus, $\deg(x\otimes e^A)=
\deg(x)+I_{c_1}(A)$, where $x\in H_*(M)$ and $A\in\Gamma$.  The
isomorphism between $\HF_*(H)$ and $\HQ_*(M)$ is defined via the
PSS-homomorphism; see \cite{PSS} or \cite{MS}.  Alternatively, it can
be obtained from a homotopy of $H$ to an autonomous $C^2$-small
Hamiltonian (under slightly more restrictive conditions than weak
monotonicity, \cite{HS}) or with a somewhat different definition of
the total Floer homology (as the limit of $\HF^{(a,\,b)}_*(H)$ as
$a\to-\infty$ and $b\to \infty$, \cite{Ono:AC}).

The \emph{quantum product} $x*y$ of two elements $H_*(M)$ is defined as 
\begin{equation}
\label{eq:qp}
x*y=\sum_{A\in\Gamma} (x*y)_A \,e^{A},
\end{equation}
where the class $(x*y)_A\in H_*(M)$ is determined by the condition that
$$
(x*y)_A\circ z=\GW^M_{A,3}(x,y,z)
$$
for all $z\in H_*(M)$. Here $\circ$ stands for the intersection index
and $\GW^M_{A,3}$ is the corresponding Gromov--Witten invariant; see
\cite{MS}. Somewhat informally, the class $(x*y)_A$ can be described
as follows. Let $J$ be a generic almost complex structure compatible
with $\omega$. Fix generic cycles $X$ and $Y$ in $M$ representing the
classes $x$ and $y$, respectively.  For the sake of simplicity, we
assume that $X$ and $Y$ are embedded.  Let $\CM$ be (the
compactification of) the moduli space of $J$-holomorphic curves
$u\colon \CP^1\to M$ in the class $A$ such that $u(0)\in X$ and
$u(1)\in Y$. Then the homology class $(x*y)_A$ is represented by the
cycle (or rather a pseudo-cycle) $\ev_\infty\colon \CM\to M$ sending
$u$ to $u(\infty)$; see, e.g., \cite{MS} for more details.

Note that $(x*y)_0=x\cap y$, where $\cap$ stands for the intersection
product and $x$ and $y$ are ordinary homology classes. Furthermore,
$$
\deg (x*y)=\deg(x)+\deg(y)-2n
$$
and
\begin{equation}
\label{eq:qp-deg}
\deg (x*y)_A=\deg(x)+\deg(y)-2n-I_{c_1}(A).
\end{equation}
Also observe that $I_\omega(A)=-\omega(A)<0$ whenever $A\neq 0$ can be
represented by a holomorphic curve. Thus, in \eqref{eq:qp}, it
suffices to limit the summation to the negative cone $I_\omega(A)\leq
0$. In particular, in the setting of Example \ref{ex:Novikov}, we can
write
$$
x*y=x\cap y+\sum_{k>0} (x*y)_k \, q^{-k}.
$$
Here, $\deg(x*y)_k=\deg(x)+\deg(y)-2n \pm 2Nk$ when $N<\infty$, with
the positive or negative sign depending on whether $M$ is positive or
negative monotone.  This sum is finite. When $N=\infty$ (i.e.,
$c_1(M)\mid_{\pi_2(M)}=0$), we have $\deg(x*y)_k=\deg(x)+\deg(y)-2n$.

The product $*$ extends to a $\Lambda$-linear, associative,
graded-commutative product on $\HQ_*(M)$. The fundamental class $[M]$
is the unit in the algebra $\HQ_*(M)$. Thus, $qx=(q[M])*x$, where
$q\in\Lambda$ and $x\in H_*(M)$, and $\deg(qx)=\deg(q)+\deg(x)$. By
the very definition of $\HQ_*(M)$, the ordinary homology $H_*(M)$ is
canonically embedded in $\HQ_*(M)$. The group of symplectomorphisms
acts on the algebra $\HQ_*(M)$ via its action on $H_*(M)$ and,
clearly, symplectomorphisms isotopic to $\id$ (in this group) act
trivially.

\begin{Example} 
\label{ex:quantum=cup}
Assume that $N>2n$. Then, as immediately follows from
  \eqref{eq:qp-deg}, the quantum product coincides with the
  intersection product: $x*y=x\cap y$. There are numerous examples of
  closed, negative monotone manifolds with $N>2n$. (Among such
  manifolds is, for instance, the hypersurface $z_0^k+\ldots+z_n^k=0$
  in $\CP^n$ with $N=k-(n+1)>2n$.) However, to the best of the authors
  knowledge, no closed, monotone manifolds with $N>n+1$ are known. In
  a similar vein, as is easy to see, $x*y=x\cap y$ when $M$ is
  negative monotone and $N\geq n$; cf.\ \cite{LO}.
\end{Example}

\begin{Example} 
\label{ex:cpn}
Let $M=\CP^n$. Then $N=n+1$ and, in the notation of Example
\ref{ex:Novikov}, $\HQ_*(\CP^n)$ is the quotient of
$\F[u]\otimes\Lambda$, where $u$ is the generator of
$H_{2n-2}(\CP^n)$, by the ideal generated by the relation
$u^{n+1}=q^{-1}[M]$; see \cite{MS}.  Thus, $u^k=u\cap \ldots\cap u$
($k$ times) when $0\leq k\leq n$ and $[\pt]*u=q^{-1}[M]$. For further
examples of calculations of quantum homology and relevant references
we refer the reader to, e.g., \cite{MS}.
\end{Example}

Finally, let us extend the map $I_\omega$ from $\Gamma$ to $\HQ_*(M)$
and $\Lambda$ by setting
$$
I_\omega(x)=\max\{I_\omega(A)\mid x_A\neq 0\}
$$
for $x=\sum x_A e^A\in\HQ_*(M)$ and $I_\omega(0)=-\infty$.  The
extension to $\Lambda$ is defined by a similar formula or can be
obtained by restricting $I_\omega$ from $\HQ_*(M)$ to $\Lambda\cong
[M]\Lambda$. The map $I_\omega$ is a valuation:
$$
I_\omega(x+y)\leq \max\{ I_\omega(x), I_\omega(y)\}
$$
and $I_\omega(\alpha x)=I_\omega(x)$ if $\alpha\in\F$ is non-zero.

\begin{Remark}
  Note in conclusion that the definition of the Novikov ring $\Lambda$
  adapted in this paper is by no means standard in the context of
  quantum homology, although it is among the most natural choices as
  far as Floer homology is concerned.  Monograph \cite{MS} offers a
  detailed discussion of a variety of choices of the Novikov ring.
\end{Remark}

\subsection{Contractible loops of Hamiltonian diffeomorphisms}
\label{sec:loops}

Our goal in this section is to recall a few, mainly well-known, facts
about contractible loops of Hamiltonian diffeomorphisms and, more
specifically, about the action of such loops on the Floer complex of a
Hamiltonian.

In what follows, \emph{a loop $\eta^t$, $t\in S^1$, is always assumed
to be based at the identity and equipped with a (homotopy type of)
contraction $\eta_s$ of $\eta_1=\eta$ to $\eta_0\equiv\id$}, i.e.,
viewed as an element of the universal covering of the identity
component in the space of loops based at $\id$. Furthermore, when the
action of a Hamiltonian on closed curves is essential, \emph{a loop
$\eta$ will also be equipped with a loop of Hamiltonians $G_t^s$
generating $\eta^t_s$}.  Since $G_t^0$ generates the identity
Hamiltonian diffeomorphism, we may assume without loss of generality
that $G_t^0\equiv 0$.

\subsubsection{Action of loops on capped closed curves and on Floer homology}
\label{sec:loops1}
With the above conventions in mind, we observe that every orbit
$y=\eta^t(p)$ of a loop $\eta$ automatically comes with a capping
arising from the contraction of $\eta$ to $\id$. Trivializing $TM$
along such a capping, we have the Maslov index of the loop $\eta$
(i.e., the Maslov index of the loop of linear symplectic
transformations $d\eta^t\mid_{\eta^t(p)}\colon T_pM\to T_pM$)
well-defined.  This index is obviously equal to zero. Furthermore,
once the loop of Hamiltonians $G_t^s$ generating $\eta^t_s$ is fixed,
the action $\CA(G):=\CA_G(\by)$ of $G$ on the capped orbit $\by$ is
independent of the initial point $p$. (This action depends on the
choice of the Hamiltonians $G^s_t$. More specifically, $\CA(G)$ is
determined by the Hamiltonians $G_t=G^1_t$ and by $\eta$ regarded as
an element of the universal covering.)

Let $\bx$ be an arbitrary capped closed curve in $M$. We can view the
capping of $x$ as a map $[0,\,1]\times S^1\to M$ sending $[0,\,
1]\times\{0\}$ and $\{0\}\times S^1$ to $p=x(0)$ and equal to $x$ on
$\{1\}\times S^1$, i.e., a family of closed curves $x_s$, $s\in
[0,\,1]$, connecting $x_0\equiv p$ with $x_1=x$ and such that
$x_s(0)=p$ for all $s$.  Then $t\mapsto \eta^t(x(t)))$ is again a
closed curve with capping $(t,s)\mapsto \eta^t_s(x_s(t))$, which we
denote by $\Phi_\eta(\bx)$ or $\Phi_G(\bx)$. We say that the loop
$\eta$ sends the capped curve $\bx$ to $\Phi_\eta(\bx)$.

To apply this observation to Floer homology, consider a one-periodic
Hamiltonian $H_t$. The time-dependent flow $\eta^t\circ \varphi^t_H$
is generated by the Hamiltonian $G\# H$ and its time-one map coincides
with $\varphi_H$ in the universal covering of the group of Hamiltonian
diffeomorphisms. Hence, the map $\Phi_G$ gives rise to a one-to-one
correspondence between capped one-periodic orbits of $H$ and those of
$G\# H$. Moreover, this map induces an isomorphism of Floer
complexes. More precisely, let $J_t$ be a time-dependent almost
complex structure on $M$. Then $\Phi_G$ induces an isomorphism between
the Floer complex of $(H, J)$ and the Floer complex of $G\# H$ with
the almost complex structure $\tilde{J}_t:=d\eta^t \circ J_t \circ
(d\eta^t)^{-1}$.  This isomorphism preserves the grading (since the
Maslov index of $\eta$ is zero) and shifts the action filtration by
$\CA(G)$. Hence, in particular, we obtain an isomorphism
$$
\HF^{(a,\,b)}_*(H)\cong \HF^{(a+\CA(G),\,b+\CA(G))}_*(G\# H).
$$
On the level of total Floer homology, where the action filtration is
essentially ignored, this isomorphism coincides with the homotopy
isomorphism $\Psi_{H,G\# H}$.  (This is equivalent to the fact that
the Seidel representation is well-defined, i.e., trivial for
contractible loops; see, e.g., \cite{MS}.)

\subsubsection{Extension of loops}
\label{sec:loop2}
The next two geometrical results discussed in this section concern the
existence and extension of loops with specific properties and are used
in the proof of Theorem \ref{thm:sdm}. Here, we mainly follow
\cite{Gi:conley}, taking into account cappings of periodic orbits.

The first of these results asserts that every capped closed curve is
an orbit of a contractible loop of Hamiltonian diffeomorphisms. More
precisely, we have

\begin{Proposition}
\label{prop:loop1}
Let $\bx$ be a capped closed curve in $M$. Then there exists a
contractible loop of Hamiltonian diffeomorphisms $\eta$ fixing
$p=x(0)$ such that $\Phi_\eta(\bx)$ is a constant curve $p$ with
trivial capping.
\end{Proposition}

\begin{proof}
As above, we can view the capping of $x$ as a smooth family of closed
curves $x_s\colon S^1\to M$, $s\in [0,1]$, connecting the constant
loop $x_0\equiv p$ to $x_1=x$. It is easy to show that there exists a
smooth family of Hamiltonians $G^s_t$ such that for every $t$, the
curve $s\mapsto x_s(t)$ is an integral curve of $G^s_t$ with respect
to $s$, i.e., $x_s=\varphi_G^s(p)$. Let $\eta^t=\varphi^1_G$ be the
time-one flow of this family, parametrized by $t\in S^1$. Then
$x=\eta^t(p)$. The family of Hamiltonians $G^s_t$ can be chosen so
that $G_{s,0}\equiv 0\equiv G_{s,1}$. Then $\eta^t$ is a loop of
Hamiltonian diffeomorphisms with $\eta^0=\id=\eta^1$. As readily
follows from the construction, the loop $\eta$ is contractible.
\end{proof}

The second result gives a necessary and sufficient condition for the
existence of an extension of a loop of local Hamiltonian
diffeomorphisms to a global loop.

Consider a loop $\eta^t$ of (the germs of) Hamiltonian diffeomorphisms
at $p\in M$ generated by $G$. In other words, the maps $\eta^t$ and
the Hamiltonian $G$ are defined on a small neighborhood of $p$ and
$\eta^t(p)=p$ for all $t\in S^1$. Then the action $\CA(G)$ and the
Maslov index $\mu(\eta)$ are introduced exactly as above with the
orbit $\bx$ taken sufficiently close to $p$. (In fact, we can set
$\bx\equiv p$ with trivial capping. Hence, $\CA(G)=\int_0^1 G_t(p)\,
dt$ and $\mu(\eta)$ is just the Maslov index of the loop $d\eta^t_p$
in $\Sp(T_pW)$.)  Note that in this case $\mu(\eta)$ need not be zero.

\begin{Proposition}
\labell{prop:loop2} Let $\eta^t$, $t\in S^1$, be a loop of germs of
Hamiltonian diffeomorphisms at $p\in M$. The following conditions are
equivalent:

\begin{itemize}

\item [(i)] the loop $\eta$ extends to a loop of global Hamiltonian
diffeomorphisms of $M$, contractible in the class of loops fixing $p$,

\item [(ii)] the loop $\eta$ is contractible in the group of germs of
Hamiltonian diffeomorphisms at $p$,

\item [(iii)] $\mu(\eta)=0$.

\end{itemize}
\end{Proposition}

\begin{proof}
The implications (i)$\Rightarrow$(ii)$\Rightarrow$(iii) are obvious.
To prove that (iii)$\Rightarrow$(ii), we identify a neighborhood of
$p$ in $M$ with a neighborhood of the origin in $\R^{2n}$. Then, as is
easy to see, the loop $\eta^t$ is homotopy equivalent to its
linearization $d\eta^t_p$, a loop of (germs of) linear maps.  By the
definition of the Maslov index, $d\eta^t$ is contractible in
$\Sp(\R^{2n})$ if and only if $\mu(\eta):=\mu(d\eta_p)=0$.

To complete the proof of the proposition, it remains to show that
(ii)$\Rightarrow$(i).  To this end, let us first analyze the case
where $\eta^t$ is $C^1$-close to the identity (and hence
contractible). Fixing a small neighborhood $U$ of $p$, we identify a
neighborhood of the diagonal in $U\times U$ with a neighborhood of the
zero section in $T^*U$. Then the graphs of $\eta^t$ in $U\times U$
turn into Lagrangian sections of $T^*U$.  These sections are the
graphs of exact forms $df_t$ on $U$, where all $f_t$ are $C^2$-small
and $f_0\equiv 0\equiv f_1$. Then we extend (the germs of) the
functions $f_t$ to $C^2$-small functions $\tf_t$ on $M$ such that
$\tf_0\equiv 0\equiv \tf_1$. The graphs of $d\tf_t$ in $T^*M$ form a
loop of exact Lagrangian submanifolds which are $C^1$-close to the
zero section. Thus, this loop can be viewed as a loop of Hamiltonian
diffeomorphisms of $M$. It is clear that the resulting loop is
contractible in the class of loops fixing $p$.

To deal with the general case, consider a contraction of $\eta$ to
$\id$, i.e., a family $\eta_s$, $s\in [0,\,1]$, of local loops with
$\eta_0\equiv \id$ and $\eta_1^t =\eta^t$.  Let
$0=s_0<s_1<\cdots<s_k=1$ be a partition of the interval $[0,\,1]$ such
that the loops $\eta_{s_i}$ and $\eta_{s_{i+1}}$ are $C^1$-close for
all $i=0,\ldots, k-1$. In particular, the loop $\eta_{s_1}$ is
$C^1$-close to $\eta_0=\id$, and thus extends to a contractible loop
$\teta_{s_1}$ on $M$. Arguing inductively, assume that a contractible
extension $\teta_{s_i}$ of $\eta_{s_i}$ has been constructed.
Consider the loop $\eta^t=\eta_{s_{i+1}}^t(\eta_{s_i}^t)^{-1}$ defined
near $p$. This loop is $C^1$-close to the identity, for
$\eta_{s_{i+1}}$ and $\eta_{s_i}$ are $C^1$-close. Hence, $\eta$
extends to a contractible loop $\tilde{\eta}$ on $M$. Then
$\teta_{s_{i+1}}^t:=\tilde{\eta}^t\teta^t_{s_i}$ is the required
extension of $\eta_{s_{i+1}}$, contractible in the class of loops
fixing $p$.
\end{proof}

\begin{Remark}
\label{rmk:loop-ext}

In the proof of Theorem \ref{thm:sdm}, we will also need the following
variant of Proposition \ref{prop:loop2}.  Assume that $\eta$ is the
germ of a loop near $p$ and the linearization of $\eta$ at $p$ is
equal to the identity, i.e., $d\eta^t_p=I$ for all $t$. (Hence,
$\mu(\eta)=0$.) Then $\eta$ extends to a loop $\teta$ of global
Hamiltonian diffeomorphisms of $M$ such that $\teta$ is contractible
in the class of loops fixing $p$ and having identity linearization at
$p$. This fact can be verified similarly to the proof of the
implication (ii)$\Rightarrow$(i).
\end{Remark}

\subsection{The mean index}
\labell{sec:mean-index}

Let $\bx$ be a capped one-periodic orbit of a Hamiltonian $H$ on $M$.
(It suffices to have $H$ defined only on a neighborhood of $x$.) The
mean index $\Delta_{H}(\bx)\in\R$ measures the sum of rotations of the
eigenvalues of $d(\varphi^t_H)_{x(t)}$ lying on the unit circle. Here
$d(\varphi^t_H)_{x(t)}$ is interpreted as a path in the group of
linear symplectomorphisms by using the trivialization of $TM$ along
$x$, associated with the capping.  (Similarly to our notation for the
action functional, we write $\Delta_{H^{(k)}}(\bx)$ when $\bx$ is
$k$-periodic.)  Referring the reader to \cite{SZ} for a precise
definition of $\Delta_{H}(\bx)$ and the proofs of its properties, we
just recall here the following facts that are used in this paper.

\begin{itemize}

\item[(MI1)] The iteration formula:
$\Delta_{H^{(k)}}(\bx^k)=k\Delta_H(\bx)$.

\item[(MI2)] Continuity: Let $\tH$ be a $C^2$-small perturbation of
$H$ and let $\by$ be a capped one-periodic orbit of $\tH$ close to
$\bx$. Then $|\Delta_H(\by)-\Delta_{\tH}(\bx)|$ is small.

\item[(MI3)] The mean index formula: Assume that $x$ is
non-degenerate. Then, as $k\to\infty$ through admissible iterations
(see Section \ref{sec:persistence}), $\MUCZ(H^{(k)},\bx^k)/k\to
\Delta_H(\bx)$.

\item[(MI4)] Relation to the Conley--Zehnder index: Let $\bx$ split
into non-degenerate orbits $\bx_1,\ldots,\bx_m$ under a $C^2$-small,
non-degenerate perturbation $\tH$ of $H$. Then
$|\MUCZ(\tH,\bx_i)-\Delta_H(\bx)|\leq n$ for all $i=1,\ldots,m$.
Moreover, these inequalities are strict when $x$ is weakly
non-degenerate; see \cite[p.\ 1357]{SZ}.  In particular, if $x$ is
non-degenerate, $|\MUCZ(H,\bx)-\Delta_H(\bx)|< n$.

\item[(MI5)] Additivity: Let $\bx_1$ and $\bx_2$ be one-periodic
orbits of Hamiltonians $H_1$ and $H_2$ on manifolds $M_1$ and,
respectively, $M_2$.  Then
$\Delta_{H_1+H_2}((\bx_1,\bx_2))=\Delta_{H_1}(\bx_1)
+\Delta_{H_2}(\bx_2)$, where $H_1+H_2$ is the naturally defined
Hamiltonian on $M_1\times M_2$.

\item[(MI6)] Action of global loops: Assume that $G$ generates a
contractible loop of Hamiltonian diffeomorphisms of $M$. Then
$\Delta_{G\# H}(\Phi_G(\bx))=\Delta_H(\bx)$.

\item[(MI7)] Action of local loops: Assume that $\bx$ is a constant
one-periodic orbit $x(t)\equiv p$ equipped with trivial capping and
that $G$ generates a loop of Hamiltonian diffeomorphisms fixing $p$
and defined on a neighborhood of $p$. Then $\Delta_{G\#
H}(\Phi_G(\bx))=\Delta_H(\bx)+2\mu$, where $\mu$ is the Maslov index
of the loop $d(\varphi^t_G)_p$.

\item[(MI8)] Index of strongly degenerate orbits: Assume that $x$ is
strongly degenerate. Then $\Delta_H(\bx)\in 2\Z$. Moreover, when
$\bx\equiv p$ is a constant orbit as in (MI7) and $H$ is defined on a
neighborhood of $p$ and generates a loop of Hamiltonian
diffeomorphisms, we have $\Delta_H(p)=2\mu$, where $\mu$ is the Maslov
index of the loop $\varphi^t_H$.

\end{itemize}

The \emph{mean index spectrum} $\CI(H)$ of $H$ is the set
$\{\Delta_H(\bx)\mid \bx\in\bPP(H)\}$. This set is closed, but, in
contrast with the action spectrum, need not have zero measure or even
be nowhere dense. The mean index spectrum modulo $2N$ is denoted by
$\CI_{2N}(H)$. Thus, $\CI_{2N}(H)=\{\Delta_H(x)\in S^1_{2N}\mid
x\in\PP(H)\}$.  Similarly to \eqref{eq:g-i1}, we have
\begin{equation}
\label{eq:g-i2}
\frac{\Delta_{H^{(k)}}(\bx)}{k}\equiv \frac{\Delta_{H^{(l)}}(\by)}{l}\mod \Q
\end{equation}
if the $k$-periodic orbit $\bx$ is geometrically identical to the
$l$-periodic orbit $\by$.

\subsection{Local Floer homology}
\labell{sec:LFH}

In this section, we briefly recall the definition and basic properties
of local Floer homology following mainly \cite{Gi:conley,GG:gap},
although this notion goes back to the original work of Floer (see,
e.g., \cite{F:witten,Fl}) and has been revisited a number of times
since then; see, e.g., \cite[Section 3.3.4]{Poz}.

Let $\bx$ be a capped isolated one-periodic orbit of a Hamiltonian
$H\colon S^1\times M\to \R$. Pick a sufficiently small tubular
neighborhood $U$ of $\gamma$ and consider a non-degenerate $C^2$-small
perturbation $\tH$ of $H$ supported in $U$.  Every (anti-gradient)
Floer trajectory $u$ connecting two one-periodic orbits of $\tH$ lying
in $U$ is also contained in $U$, provided that $\|\tH-H\|_{C^2}$ and
$\supp(\tH-H)$ are small enough.  Thus, by the compactness and gluing
theorems, every broken anti-gradient trajectory connecting two such
orbits also lies entirely in $U$. Similarly to the definition of local
Morse homology, the vector space (over $\F$) generated by one-periodic
orbits of $\tH$ in $U$ is a complex with (Floer) differential defined
in the standard way. The continuation argument (see, e.g.,
\cite{MS,SZ}) shows that the homology of this complex is independent
of the choice of $\tH$ and of the almost complex structure. We refer
to the resulting homology group $\HF_*(H,\bx)$ as the \emph{local
Floer homology} of $H$ at $\bx$.

\begin{Example}
Assume that $\bx$ is non-degenerate and $\MUCZ(\bx)=k$.  Then
$\HF_l(H,\bx)=\F$ when $l=k$ and $\HF_l(H,\bx)=0$ otherwise.
\end{Example}

In the rest of this section, we list the basic properties of local
Floer homology that are essential for what follows.

\begin{enumerate}
\item[(LF1)] Let $H^s$, $s\in [0,\, 1]$, be a family of Hamiltonians
such that $x$ is a \emph{uniformly isolated} one-periodic orbit for
all $H^s$, i.e., $x$ is the only periodic orbit of $H_s$, for all $s$,
in some open set independent of $s$. Then $\HF_*(H^s,\bx)$ is constant
throughout the family: $\HF_*(H^0,\bx)=\HF_*(H^1,\bx)$.
\end{enumerate}

Local Floer homology spaces are building blocks for filtered Floer
homology. Namely, essentially by definition, we have the following:

\begin{enumerate}
\item[(LF2)] Assume that all one-periodic orbits $x$ of $H$ are
isolated and $\HF_k(H,\bx)=0$ for some $k$ and all $\bx$. Then
$\HF_k(H)=0$. Moreover, let $M$ be rational and closed and let $c\in
\R$ be such that all capped one-periodic orbits $\bx_i$ of $H$ with
action $c$ are isolated. (As a consequence, there are only finitely
many orbits with action close to $c$.) Then, if $\eps>0$ is small
enough,
$$
\HF_*^{(c-\eps,\,c+\eps)}(H)=\bigoplus_i \HF_*(H,\bx_i).
$$

\end{enumerate}

Just as ordinary Floer homology, the local Floer homology is
completely determined by the time-one map generated by $H$ viewed as
an element of the universal covering of the group of Hamiltonian
diffeomorphisms:

\begin{enumerate}
\item[(LF3)] Let $\varphi^t_G$ be a contractible loop of Hamiltonian
diffeomorphisms of $M$. Then
$$
\HF_*(G\# H,\Phi_G(\bx)) = \HF_*(H,\bx)
$$
for every isolated one-periodic orbit $x$ of $H$.
\end{enumerate}

Furthermore, the K\"unneth formula holds for local Floer homology:

\begin{enumerate}
\item[(LF4)] Let $\bx_1$ and $\bx_2$ be capped one-periodic orbits of
Hamiltonians $H_1$ and $H_2$ on, respectively, symplectic manifolds
$M_1$ and $M_2$. Then $\HF_*\big(H_1+H_2, (\bx_1,\bx_2)\big)
=\HF_*(H_1,\bx_1)\otimes \HF_*(H_2,\bx_2)$, where $H_1+H_2$ is the
naturally defined Hamiltonian on $M_1\times M_2$.
\end{enumerate}

By definition, the \emph{support} of $\HF_*(H,\bx)$ is the collection
of integers $k$ such that $\HF_k(H,\bx)\neq 0$. Clearly, the group
$\HF_*(H,\bx)$ is finitely generated and hence supported in a finite
range of degrees. The next observation, providing more precise
information on the support of $\HF_*(H,\bx)$, is an immediate
consequence of (MI4).

\begin{enumerate}
\item[(LF5)] The group $\HF_*(H,\bx)$ is supported in the range
$[\Delta_H(\bx),\, \Delta_H(\gamma)+2n]$. Moreover, when $x$
is weakly non-degenerate, the support is contained in the open
interval $(\Delta_H(\bx),\, \Delta_H(\bx)+2n)$.
\end{enumerate}

As is easy to see from the definition of local Floer homology, $H$
need not be a function on the entire manifold $M$ -- it is sufficient
to consider Hamiltonians defined only on a neighborhood of $x$. For
the sake of simplicity, we focus on the particular case where
$x(t)\equiv p$ is a constant orbit equipped with trivial
capping. Thus, $dH_t(p)=0$ for all $t\in S^1$. Then (LH1), (LH4) and
(LF5) still hold, and (LF3) takes the following form:

\begin{enumerate}
\item[(LF6)] Let $\varphi^t_G$ be a loop of Hamiltonian
diffeomorphisms defined on a neighborhood of $p$ and fixing $p$. Then
$$
\HF_*(G\# H,p) = \HF_{*+2\mu}(H,p),
$$
where $\mu$ is the Maslov index of the loop
$t\mapsto d(\varphi^t_G)_p\in \Sp(T_pM)$.
\end{enumerate}

\section{The Conley conjecture type results}
\label{sec:conley-pfs}

Our goal in this section is to prove three Conley conjecture type
results -- Theorems \ref{thm:conley}, \ref{thm:N/n}, and
\ref{thm:sdm-conley} -- and to illustrate our point that the
arguments, in fact, concern the action or index spectra of the
Hamiltonian in question. The proofs rely on Theorem \ref{thm:sdm},
which is established in Section \ref{sec:p-o-sdm}.

\subsection{The normalized action and mean index spectra}
\label{sec:norm-spectr}
The most direct way to obtain information about the number of periodic
orbits via the action or mean index spectrum $\CS_{\lambda_0}(H)$ and
$\CI_{2N}(H)$, discussed in Sections \ref{sec:capped-per-orb} and
\ref{sec:mean-index}, is by simply observing that the number of (not
necessarily simple) geometrically distinct $k$-periodic orbits of $H$
is bounded from below by $\big|\CS_{\lambda_0}\big(H^{(k)}\big)\big|$
and $\big|\CI_{2N}\big(H^{(k)}\big)\big|$. (Here we allow the
cardinality $|\cdot|$ to assume infinite value, but we do not
distinguish between different kinds of infinity.) Clearly, $H$ has
infinitely many geometrically distinct periodic orbits whenever
\begin{equation}
\label{eq:eq-spec}
\limsup \big|\CS_{\lambda_0}\big(H^{(k)}\big)\big|=\infty
\text{ or }
\limsup \big|\CI_{2N}\big(H^{(k)}\big)\big|=\infty
\text{ as } k\to\infty.
\end{equation}

The spectra of the iterations $H^{(l)}$ and $H^{(kl)}$ are related via
natural maps
$$
\CS_{\lambda_0}\big(H^{(l)}\big)\to k\CS_{\lambda_0}\big(H^{(l)}\big)
\subset \CS_{\lambda_0}\big(H^{(kl)}\big).
$$
When $\big|\PP^\infty(H)\big|<\infty$, the inclusions
$k\CS_{\lambda_0}\big(H^{(l)}\big) \subset
\CS_{\lambda_0}\big(H^{(kl)}\big)$ ``stabilize'', i.e.,
$\CS_{\lambda_0}\big(H^{(kd)}\big)= k\CS_{\lambda_0}\big(H^{(d)}\big)$ 
for all $k$,
where $d$ is the least common multiple of simple periods of $H$.
Thus, we obtain the following criterion: \emph{ Assume that $H$ has
  finitely many periodic orbits. Then, for every $k_0>0$ there exists
  $l>0$, divisible by $k_0$, such that
  $k\CS_{\lambda_0}\big(H^{(l)}\big) =
  \CS_{\lambda_0}\big(H^{(kl)}\big)$ for every $k$.}  Clearly, the
same considerations apply when the action spectrum is replaced by the
mean index spectrum $\CI_{2N}\big(H^{(k)}\big)$.

Alternatively, to account for the role of recapping and iterations, it
is sometimes convenient to interchange taking the $\limsup$ and
cardinality in \eqref{eq:eq-spec} and put together the spectra of all
iterations $H^{(k)}$ as follows.

Let $a\in (0,\infty]$. Consider the set of infinite sequences
$a_l,\,a_{2l}=2a_l,\, a_{3l}=3a_l,\ldots$ in $S^1_a=\R/a\Z$, for all
integers $l>0$. Let us call two sequences $\{a_{il}\}$ and
$\{a'_{il'}\}$ equivalent if they coincide on their common domain of
indices or, equivalently, $a_d=a'_d$, where $d$ is the least common
multiple of $l$ and $l'$. The set of such sequences, up to this
equivalence relation, is in one-to-one correspondence with $\R/a\Q$
via the map sending $\{a_{il}\}$ to $a_{il}/(il)$. Here, when
$a=\infty$, the quotient $\R/a\Q$ is taken to be $\R$.

Let now $H$ be a one-periodic Hamiltonian on a rational manifold
$M$. Then every capped $k$-periodic orbit $\bx$ of $H$ gives rise to a
sequence $a_{ik}=\CA_{H^{(ik)}}(\bx^i)$.  The resulting element of
$\R/\lambda_0\Q$, which we will call the normalized action of $H$ on
$\bx$, is independent of capping and, by \eqref{eq:g-i1}, two
geometrically identical periodic orbits give rise to the same
element. The set $\CS_{\lambda_0}^{\infty}(H)\subset \R/\lambda_0\Q$
of all normalized actions of $H$ will be called the \emph{normalized
action spectrum} of $H$. The \emph{normalized mean index spectrum} of
$H$, denoted by $\CI^{\infty}(H)\subset \R/2N\Q=\R/\Q$, is defined in
a similar fashion. Finally note that the assumption that $H$ is
one-periodic is not essential and the definitions extend to
$k$-periodic Hamiltonians in an obvious way.  It is clear from
\eqref{eq:g-i1} and \eqref{eq:g-i2} that
$$
\big|\PP^\infty(H)\big|\geq
\limsup \big|\CS_{\lambda_0}\big(H^{(k)}\big)\big|\geq 
\big|\CS_{\lambda_0}^{\infty}(H)\big|
$$
and 
$$
\big|\PP^\infty(H)\big|\geq
\limsup \big|\CI_{2N}\big(H^{(k)}\big)\big|\geq
\big|\CI^{\infty}(H)\big|.
$$
Although these lower bounds are admittedly indiscriminating tools for
keeping track of periodic orbits, they do naturally arise in the
proofs of Theorems \ref{thm:N/n} and \ref{thm:sdm-conley}. Note also
the inclusions
$\CS_{\lambda_0}^{\infty}\big(H^{(k)}\big)\subset\CS_{\lambda_0}^{\infty}(H)$
and $\CI^{\infty}\big(H^{(k)}\big)\subset\CI^{\infty}(H)$ which hold
for any integer $k>0$.

\subsection{Proof of Theorem \ref{thm:sdm-conley}}

The first assertion (i) is a consequence of the following result

\begin{Proposition}
\label{prop:sdm-conley2}
Let $M$ and $H$ be as in Theorem \ref{thm:sdm-conley}(i) and all
periodic orbits of $H$ are isolated.  Then for every $l>0$ divisible
by $k_0$ there exists $k>0$ such that
$k\CS_{\lambda_0}\big(H^{(l)}\big) \neq
\CS_{\lambda_0}\big(H^{(kl)}\big)$, i.e., the action spectrum of
$H^{(k)}$ never stabilizes.
\end{Proposition}

\begin{proof}[Proof of Proposition \ref{prop:sdm-conley2}]
  Let us focus on the case where $\lambda_0<\infty$. (The case of
  $\lambda_0=\infty$, i.e., $\omega\mid_{\pi_2(M)}=0$ is treated in a
  similar way. Moreover, then $\limsup
  \big|\CS\big(H^{(k)}\big)\big|=\infty$.)  Furthermore, let us
  normalize $H$ so that the action on the symplectically
  non-degenerate maximum is equal to zero.

  Arguing by contradiction, assume that there exists $l>0$, divisible
  by $k_0$, such that $k\CS_{\lambda_0}\big(H^{(l)}\big) =
  \CS_{\lambda_0}\big(H^{(kl)}\big)$ for every $k$.  The Hamiltonian
  $F=H^{(l)}$ still has a symplectically degenerate maximum (with
  action value equal to zero), for an iteration of a symplectically
  non-degenerate maximum is again a symplectically non-degenerate
  maximum; see \cite{GG:gap}.  Thus, Theorem \ref{thm:sdm} applies to
  $F$. Finally, for the sake of simplicity, let us assume that $F$ is
  one-periodic, for the period can always be altered via a
  reparametrization.

  The action spectrum $\CS_{\lambda_0}(F)$ is finite since the
  periodic orbits of $H$, and hence of $F$, are isolated. Let
$$
\CS_{\lambda_0}(F)=\big\{a_1,\ldots,a_r\big\}\cup\lambda_0\big\{p_1/q_1,\ldots,p_s/q_s\big\},
$$
where the first $r$ points $a_i$ are irrational modulo $\lambda_0$,
the last $s$ points $p_j/q_j$ are rational and we require $p_j$ and
$q_j$ to be mutually prime. Let $d$ be the least common multiple of
$q_1,\ldots, q_s$.  Then, due to our assumptions,
$$
\CS_{\lambda_0}\big(F^{(k)}\big)=k\CS_{\lambda_0}(F)
\subset
\big\{ka_1,\ldots,ka_r\big\}\cup\frac{\lambda_0}{d}\Z
$$ 
modulo $\lambda_0$. By Theorem \ref{thm:sdm}, for every sufficiently
small $\eps>0$, there exists $k_\eps>0$ such that for every
$k>k_\eps$, at least one of the points from
$\CS_{\lambda_0}\big(F^{(k)}\big)$ is in the arc $(0,\,\eps)\subset
S^1$. When $\eps<1/d$, this is one of the points $ka_i$.  Denote by
$p_i(m,\eps)$ the probability that the point $ka_i$ with $k_\eps<k\leq
m$ is in the arc $(0,\,\eps)$. In other words, $p_i(m,\eps)$ is the
number of integers $k\in (k_\eps,\,m]$, divided by $m-k_\eps$, such
that $ka_i\in (0,\,\eps)$.  Since the rotation of $S^1$ in $a_i$ is
ergodic, $p_i(m,\eps)\to \eps$ as $m\to\infty$. Summing up over all
$i=1,\ldots,r$, we conclude that the number of integers $k\in
(k_\eps,\,m]$ with $ka_i\in (0,\,\eps)$ for at least one
$i=1,\ldots,r$ does not exceed $(m-k_\eps)r\eps+o(m-k_\eps)$.
Therefore, when $\eps<1/r$ and $m$ is large enough, there exists an
integer $k\in (k_\eps,\,m]$ such that $ka_i\not\in (0,\,\eps)$ for all
$i$. This contradicts Theorem~\ref{thm:sdm}.
\end{proof}

\begin{Remark} We conjecture that 
$\limsup \big|\CS_{\lambda_0}\big(H^{(k)}\big)\big|=\infty$ as
$k\to\infty$, under the hypotheses of Proposition
\ref{prop:sdm-conley2}. It is easy to see that this is true when
$\lambda_0=\infty$ and the argument above falls just short of showing
that this is true in general.
\end{Remark}

Returning to the proof of Theorem \ref{thm:sdm-conley}, let us first
assume that $c_1(M)\mid_{\pi_2(M)}=0$. In this case, the mean index of
a periodic orbit is independent of capping and we suppress the capping
in the notation. By Theorem \ref{thm:sdm} and (LF2), for every
sufficiently large $k$, the Hamiltonian $H$ has a $k$-periodic orbit
$y$ with $\Delta_{H^{(k)}}(y)\in [1,\, 2n+1]$. When $k$ is prime, $y$
is either simple or the $k$-th iteration of a one-periodic orbit $z$
of $H$. The latter is impossible when $k$ is large, since for every
one-periodic orbit $z$ either $\Delta_{H^{(k)}}(z^k)\equiv 0$ or
$|\Delta_{H^{(k)}}(z^k)|\to \infty$ as $k\to\infty$, and there are
only finitely many one-periodic orbits. (Observe also that here, in
contrast with the general case, the assumption that $M$ is rational is
used only to ensure that $M$ satisfies the requirements of Theorem
\ref{thm:sdm}, where it plays a purely technical role.)

Finally, the remaining case where $\omega\mid_{\pi_2(M)}=0$ (i.e.,
$\lambda_0=\infty$) is treated in a similar fashion, but with action
used in place of the mean index. This is a standard argument; see,
e.g., \cite{Gi:conley,Hi}.

\subsection{Proof of Theorem \ref{thm:conley}} Here, as 
in \cite{Gi:conley,GG:gap,Hi}, we consider two cases: the case where
$H$ has a symplectically non-degenerate maximum and the case where all
one-periodic orbits of $H$ have non-zero mean index. The assertion in
the former case follows from Theorem \ref{thm:sdm-conley}. In the
latter case, since the mean index of an orbit is independent of
capping, the support of local Floer homology of any periodic orbit
shifts away from the interval $[0,\,2n]$ as the order of iteration
grows. Then the theorem follows again by the standard argument; see
\cite{Gi:conley,Hi,SZ}.

\subsection{Proof of Theorem \ref{thm:N/n}} The theorem is an immediate 
consequence of the following result:

\begin{Proposition}
\label{prop:N/n2}
Let $M$ be a closed, rational, weakly monotone manifold with $2N> 3n$
and let $H$ be a Hamiltonian on $M$ with finitely many geometrically
distinct periodic orbits. Then $\big|\CI^{\infty}(H)\big|\geq \lceil
N/n \rceil$.
\end{Proposition}

\begin{proof}[Proof of Proposition \ref{prop:N/n2}] 
  Since $H$ has finitely many periodic orbits, the Hamiltonian
  $F=H^{(k_0)}$, for a suitably chosen iteration $k_0>0$, has the
  following properties:
\begin{itemize}
\item $\CI_{2N}\big(F^{(k)}\big)\subset k\CI_{2N}(F)$ for all $k>0$;

\item $\CI_{2N}(F)\cap (\Q/\Z)$ contains at most one point $\{0\}$,
i.e., whenever the mean index of a periodic orbit of $F$ is rational,
this index is integer and divisible by $2N$;

\item $(\CI_{2N}(F)-\CI_{2N}(F))\cap (\Q/\Z)$ contains at most one
point $\{0\}$, i.e., whenever the difference of the mean indices of
two periodic orbits of $F$ is rational, this difference is an integer
divisible by $2N$.
\end{itemize}
It suffices to show that $\big|\CI^{\infty}(F)\big|\geq \lceil N/n
\rceil$.  For the sake of simplicity, let us assume again that $F$ is
one-periodic (rather than $k_0$-periodic) -- this can always be
achieved via a change of time.

We break the collection of capped one-periodic orbits of $F$ into two
groups: those with integer mean indices (divisible by $2N$) and the
remaining capped orbits -- those with irrational mean indices. These
groups do not change when $F$ is replaced by any of its iterations
$F^{(k)}$ due to the first condition in the description of $F$.  By
(LF2), for every $k$, there exists a $k$-periodic orbit $\bx$ such
that $\HF_{2n}(F^{(k)},\bx)\neq 0$. We claim that $\bx$ is necessarily
in the second group, i.e., $\Delta_{F^{(k)}}(\bx)\not\in
2N\Z$. Indeed, otherwise, we would have $\Delta_{F^{(k)}}(\bx)=2Nl$
for some $l\in \Z$. By (LF5), the local Floer homology of $\bx$ is
supported in the interval $[2Nl,\, 2Nl+2n]$. When $l\neq 0$, this
interval does not contain $2n$, and hence $\HF_{2n}(F^{(k)},\bx)=
0$. If $l=0$, the orbit $\bx$ is a symplectically degenerate maximum
of $F^{(k)}= H^{(kk_0)}$, and thus $H$ has infinitely many
geometrically distinct periodic orbits by Theorem
\ref{thm:sdm-conley}(i), which contradicts our basic assumption.
  
  Let $\CS_{2N}(F)\ssminus \{0\}=\{\Delta_1,\ldots,\Delta_r\}$.  By
  our choice of $k_0$ (the last two conditions),
  $\big|\CI^{\infty}(F)\big|\geq r$. Since $\Delta_i$ is irrational,
  the point $k\Delta_i$ enters the arc $[n,\,3n]\subset
  S^1_{2N}=\R/2N\Z$ of length $2n$ with probability $n/N$. To be more
  precise, for every $k>0$ the number of $j$ such that $1\leq j\leq k$
  and $j\Delta_i\in [n,\,3n]$ is equal to $kn/N$, up to a constant
  independent of $k$. As is observed above, for every $k>0$, at least
  one of the points $k\Delta_1,\ldots,k\Delta_r$ is within the arc
  $[n,\,3n]$. Hence, $r\geq \lceil N/n \rceil$.
\end{proof}

\begin{Remark}
\label{rmk:resonances}
A minor modification of this argument shows that whenever $H$ has
finitely many periodic orbits and $N\geq n+1$, the mean indices of the
periodic orbits of $H$ must satisfy a resonance relation modulo
$2N$. The circle of questions concerning these resonance relations is
explored in more detail in \cite{GK}.
\end{Remark}

\section{The bounded gap theorem} 
\label{sec:gap-pfs}

In this section, we establish the bounded gap theorem (Theorem
\ref{thm:gap1}), following closely the line of reasoning from
\cite{Gi:conley,GG:gap} where a similar result is proved for
symplectically aspherical manifolds. The proof of Theorem
\ref{thm:gap1} relies on Theorem \ref{thm:sdm}, established in Section
\ref{sec:p-o-sdm}, and on two properties of Floer homology: the
stability of filtered Floer homology and the persistence of local
Floer homology under iterations, which we discuss first.

\subsection{Stability of Floer homology}
\labell{sec:stab} 
The stability of filtered Floer homology asserts that the part of
filtered Floer homology that is stable under perturbations of the
action interval is also stable with respect to variations of the
Hamiltonians; cf.\ \cite{BC,Ch,GG:gap}.

To be more precise, consider Hamiltonians $K$ and $F$ on a closed
weakly monotone symplectic manifold $M$.  Set
$$
E^+=\int_0^1 \max_M F_t\,dt\quad\text{and}\quad
E^-=-\int_0^1 \min_M F_t\,dt
$$
so that $\|F\|=E^++E^-$ is the Hofer energy of $F$. Furthermore, let
$$
E^+_{0}=\max\{E^+,0\}\quad\text{and}\quad
E^-_{0}=\max\{E^-,0\}\quad\text{and}\quad
E_0(F)=E_{0}^+ +E_{0}^-.
$$
Note that $E^\pm_0=E^\pm$ and $E_0=\| F\|$ when $F$ is normalized 
or, more generally, $\min_M F_t\leq 0\leq \max_M F_t$.

Let $a<b$ be such that the end points of the intervals $(a,\,b)$ and
$(a+E_0,\,b+E_0)$ are outside $\CS(K)$ while the end points of
$(a+E_0^+\,,b+E_0^+)$ are outside $\CS(K\#F)$. These requirements are
met by almost all $a$ and $b$.

Furthermore, consider the following commutative diagram:
\begin{displaymath}
\xymatrix
{
\HF^{(a,\,b)}_*(K) \ar[r]^-\Psi \ar[dr]_{\kappa}
& \HF^{(a+E^+_0,\,b+E^+_0)}_*(K\# F) \ar[d] \\
& \HF^{(a+E_0,\,b+E_0)}_*(K)
}
\end{displaymath}
where $\kappa$ is the natural ``quotient-inclusion'' map, the
homomorphism $\Psi$ is induced by the linear homotopy from $K$ to $K\#
F$, and the vertical arrow is induced by the linear homotopy from $K\#
F$ to $K$. We obviously have

\begin{Lemma}[Stability of filtered Floer homology, \cite{GG:gap}]
\labell{lemma:stab}
Assume that $\kappa(x)\neq 0$, where
$x\in \HF^{(a,\,b)}_*(K)$. Then $\Psi(x)\neq 0$. In particular,
$\HF^{(a+E_0^+,\,b+E_0^+)}_*(K\# F)\neq 0$, whenever $\kappa\neq 0$.
\end{Lemma}

\begin{Remark}
  It is worth mentioning that the assertion of Lemma \ref{lemma:stab}
  is meaningful only when the length of the interval $(a,\, b)$
  exceeds the size of the perturbation: if the intervals $(a,\,b)$ and
  $(a+E_0,\,b+E_0)$ do not overlap, the quotient-inclusion map
  $\kappa$ is necessarily zero.
\end{Remark}

\subsection{Persistence of local Floer homology}
\labell{sec:persistence}

The second ingredient needed in the proof of Theorem \ref{thm:gap1} is
the persistence of local Floer homology under iterations.

Let $M$ be an arbitrary symplectic manifold and let $x$ be an isolated
one-periodic orbit of a Hamiltonian $H\colon S^1\times M\to \R$.
Recall that a positive integer $k$ is an \emph{admissible iteration}
of $\varphi=\varphi_H$ (with respect to $x$) if $\lambda^k\neq 1$ for
all eigenvalues $\lambda\neq 1$ of $d\varphi_p\colon T_pM\to T_pM$,
where $p=x(0)$. In other words, $k$ is admissible if and only if
$d\varphi_p^k$ and $d\varphi_p$ have the same generalized eigenvectors
with eigenvalue one. For instance, when no eigenvalue $\lambda\neq 1$
is a root of unity (e.g., all eigenvalues are equal to 1), all $k> 0$
are admissible. It is easy to see that admissible iterations form a
quasi-arithmetic sequence.

\begin{Theorem}[Persistence of local Floer homology, \cite{GG:gap}]
\labell{thm:persist-lf}

Assume that $x$ is an isolated one-periodic orbit of $H$.  Then the
$k$-periodic orbit $x^k$ is also isolated, whenever $k$ is admissible,
and the local Floer homology groups of $H$ and $H^{(k)}$ coincide up
to a shift of degree:
\begin{equation}
\labell{eq:case2}
\HF_*\big(H^{(k)} ,\bx^k\big)=\HF_{*+s_k}(H,\bx)
\quad\text{for some $s_k$.}
\end{equation}
Furthermore, $\lim_{k\to \infty}s_k/k=\Delta_H(\bx)$, provided that
$\HF_*(H,\bx)\neq 0$ and hence the shifts $s_k$ are uniquely
determined by \eqref{eq:case2}.  Moreover, when $\Delta_H(\bx)=0$ and
$\HF_{2n}(H,\bx)\neq 0$, the orbit $x$ is strongly degenerate, all
$s_k$ are zero, and $\bx$ is a symplectically degenerate maximum.

\end{Theorem}

This theorem is proved in \cite{GG:gap} under the additional
assumption that $M$ is symplectically aspherical. However, as was
already pointed out in that paper, the argument is essentially local
and goes through for a general symplectic manifold. Namely, by the
results of Section \ref{sec:loops}, the statement reduces to the case
where $\bx$ is a constant orbit $p$ with trivial capping. The rest of
the proof concerns only the behavior of $H$ on a neighborhood of $p$;
see \cite{GG:gap}.

\subsection{Proof of Theorem \ref{thm:gap1}}

Similarly to its predecessors from \cite{Gi:conley,GG:gap}, the
argument is based on the analysis of two cases.  Namely, since
$\HF_{2n}(H)\neq 0$, there exists a capped one-periodic orbit $\bx$ of
$H$ with $\HF_{2n}(H,\bx)\neq 0$. Thus, $\Delta_H(\bx)\geq 0$. The
first, ``degenerate'', case is where $\Delta_H(\bx)=0$, i.e., $\bx$ is
a symplectically degenerate maximum, while in the second,
``non-degenerate'', case $\Delta_H(\bx)>0$.  (Since, in general, $\bx$
is not unique, the two cases are not mutually exclusive for a given
Hamiltonian $H$.)

\subsubsection{The ``degenerate'' case: $\HF_{2n}(H,\bx)\neq 0$ and 
$\Delta_H(\bx)=0$} 
This is the case where $\bx$ is a symplectically degenerate maximum of
$H$. By Theorem \ref{thm:sdm}, for every sufficiently small $\eps>0$
there exists an integer $k_\eps>0$ such that for all $k>k_\eps$ we
have
$$
\HF_{2n+1}^{(kc,\,kc+\eps)} \big(H^{(k)}\big)\neq 0,
$$
where $c=A_H(\bx)$. Hence, $H$ has a $k$-periodic orbit $\bz_k$ with
$$
0<\A_{H^{(k)}}(\bz_k)-\A_{H^{(k)}}(\bx^k)<\eps
$$
and
$$
1\leq \Delta_{H^{(k)}}(\bz_k)-\Delta_{H^{(k)}}(\bx^k)\leq 2n+1.
$$
Note that once $\eps$ is smaller than the rationality constant $\lambda_0$ of
$M$, the orbit $\bz_k$ cannot be a recapping of $\bx^k$.
It remains to set $\nu_i$=$k_\eps+i$ and $\by_i=\bz_{k_\eps+i}$.

\subsubsection{The ``non-degenerate'' case: $\HF_{2n}(H,\bx)\neq 0$ and
  $\Delta_H(x)>0$}
In this case, $x$ is weakly non-degenerate; see \cite{GG:gap}. Furthermore,
it is easy to see that there exists an infinite, quasi-arithmetic
sequence $l_1<l_2<\ldots$ of admissible iterations such that
$$
l_i+1<l_{i+1}
$$
and the local Floer homology jumps from $l_i$ to $l_{i+1}$, i.e.,
\begin{equation}
\labell{eq:l}
0\neq \HF_*\big(H^{(l_i)},\bx^{l_i}\big)\neq \HF_*\big(H^{(l_i+1)},\bx^{l_i+1}\big).
\end{equation}

Note that here the group $\HF_*\big(H^{(l_i)},\bx^{l_i}\big)$ is
non-trivial by Theorem \ref{thm:persist-lf} and, by (LF5), supported
in the interval $(l_i\Delta_H(\bx),\, l_i\Delta_H(\bx)+2n)$, while the
group $\HF_*\big(H^{(l_i+1)},x^{l_i+1}\big)$ may possibly be trivial.

Let us normalize $H$ by requiring it to have zero mean. Then $E_0=\|
H\|$ as has been pointed out above. Set $c=\CA_H(\bx)$. Thus,
$\CA_{H^{(k)}}(\bx^k)=kc$ for all $k$.  Recall also that, by the
hypotheses of the theorem, we have
\begin{itemize}
\item[(i)] $E_0=\|H\|<\lambda_0$, where $\lambda_0>0$ is the
rationality constant of $M$, or
\item[(ii)] $N\geq 2n$, where $N$ is the minimal Chern number of $M$.
\end{itemize}

Fix $\alpha>E_0>0$ and
$\beta>0$. When condition (i) is satisfied, we also require that
\begin{equation}
\labell{eq:alpha+beta}
\max\{\alpha,\beta+E_0\}<\lambda_0.
\end{equation}

The $i$th entry in the sequence $\nu_i$ is equal either to $l_i$ or to
$l_i+1$.  Assume that $\nu_1,\ldots, \nu_{i-1}$ and the periodic
orbits $\by_1,\ldots,\by_{i-1}$ have been chosen. Our goal is then to
find a $\nu_i$-periodic orbit $\by=\by_i$ with either $\nu_i=l_i$ or
$\nu_i=l_{i}+1$ satisfying the requirements of Theorem
\ref{thm:gap1}. (The first orbit $y_1$ and the period $\nu_1$ are
chosen in a similar fashion.) Note that, since $x$ is one-periodic,
$\by_i$ and $\bx^{\nu_i}$ are geometrically distinct if and only if
$y_i\neq x^{\nu_i}$, i.e., $\by_i$ is not a recapping of
$\bx^{\nu_i}$.

Set $l=l_i$ to simplify the notation, and $a=lc-\alpha$ and
$b=lc+\beta$.  Fix $m$ such that $\HF_m\big(H^{(l)},\bx^{l}\big)\neq
0$, but $\HF_m\big(H^{(l+1)},\bx^{l+1}\big)= 0$.  Under the above
assumptions, $\by_i$ and $\nu_i$ are chosen differently in each of the
following three cases.

\emph{Case 1: $\HF_m^{(a,\,b)}\big(H^{(l)}\big)=0$.} It is easy to see
that in this case $H$ has an $l$-periodic orbit $\by$, killing the
contribution of $\HF_m\big(H^{(l)},\bx^{l}\big)$ to
$\HF_m^{(a,\,b)}\big(H^{(l)}\big)$, such that
$|\Delta_{H^{(l)}}(\by)-m|\leq n+1$ and the action
$\CA_{H^{(l)}}(\by)\neq kc$ is in the interval $(a,\,b)$.  It is clear
that the action and mean index gaps for $\bx^{l}$ and $\by$ are
bounded from above by $\max\{\alpha,\beta\}$ and, respectively, $2n+1$
and the action gap is strictly positive. Hence, in particular,
$\by\neq \bx^k$.

If condition (i) is satisfied, $\by$ is not a non-trivial recapping of
$\bx^l$ due to \eqref{eq:alpha+beta}. Under condition (ii), the same
is true since the mean index gap does not exceed $2n+1$ which is
smaller than $2N$. (Here it would be sufficient to assume $N>n$ in
(ii).) Thus $y\neq x^l$. We set $\nu_i=l$ and $\by_i=\by$.

\emph{Case 2: $\HF_m^{(a+E^+_0,\,b+E^+_0)}\big(H^{(l+1)}\big)\neq 0$.}
In this case, there exists an $(l+1)$-periodic orbit $\by$ with action
in the interval $(a+E^+_0,\,b+E^+_0)$ and
$\HF_m\big(H^{(l+1)},\by\big)\neq 0$.  We set $\nu_i=l+1$ and
$\by_i=\by$. To verify the requirements of the theorem, we first note
that
\begin{equation}
\labell{eq:action-bound}
|\CA_{H^{(l+1)}}(\bx^{l+1})-\CA_{H^{(l+1)}}(\by)|
\leq \max \{\alpha-E^+_0,\beta+E^+_0\}.
\end{equation}
Furthermore, since $|\Delta_{H^{(l+1)}}(\by)-m|\leq n$ and
$|\Delta_{H^{(l)}}(\bx^l)-m|\leq n$, we have
\begin{equation}
\labell{eq:index-bound}
|\Delta_{H^{(l+1)}}(\bx^{l+1})-\Delta_{H^{(l+1)}}(\by)| \leq
\Delta_H(\bx)+2n< 4n.
\end{equation}
The last inequality in \eqref{eq:index-bound} follows from the fact
that $\HF_{2n}(H,\bx)\neq 0$, and hence, by (LF5), $\Delta_H(x)< 2n$
since $x$ is weakly non-degenerate. The inequalities
\eqref{eq:action-bound} and \eqref{eq:index-bound} give upper bounds
on the action and mean index gaps.

To show that $y\neq x^{l+1}$, i.e., $\by$ is not a recapping of
$\bx^{l+1}$, we argue as follows. First note that $\by\neq\bx^{l+1}$,
since $\HF_m\big(H^{(l+1)},\bx^{l+1}\big)= 0$ while
$\HF_m\big(H^{(l+1)},\by\big)\neq 0$.  Thus, it remains to prove that
$\by$ cannot be a non-trivial recapping of $\bx^{l+1}$. When condition
(i) is satisfied this follows immediately from \eqref{eq:alpha+beta}
and \eqref{eq:action-bound}. If condition (ii) holds, this follows
from \eqref{eq:index-bound}.

\emph{Case 3: $\HF_m^{(a,\,b)}\big(H^{(l)}\big)\neq 0$, but
$\HF_m^{(a+E,\,b+E)}\big(H^{(l+1)}\big)= 0$.} First note that
throughout the entire argument we might have assumed that $H_0\equiv 0
\equiv H_1$, and thus $H^{\# k}$ is also one-periodic for all $k$. Now
we can identify one-periodic orbits of $H^{\# k}$ with $k$-periodic
orbits of $H$ and the Floer homology of $H^{\# k}$ and of $H^{(k)}$.
Using stability of filtered Floer homology as in Section
\ref{sec:stab} with $K=H^{\# l}$ and $F=H$, we see that the
quotient--inclusion map
$$
\kappa\colon \HF_m^{(a,\,b)}\big(H^{\# l}\big)\to
\HF_m^{(a+E,\,b+E)}\big(H^{\# l}\big)
$$
is necessarily zero, for $\HF_m^{(a+E^+_0,\,b+E^+_0)}\big(H^{\#
(l+1)}\big)= 0$. Since $\HF_m^{(a,\,b)}\big(H^{\# l}\big)\neq 0$, we
infer by a simple exact sequence argument that
$$
\HF_m^{(a,\,a+E)}\big(H^{(l)}\big)=\HF_m^{(a,\,a+E)}\big(H^{\# l}\big)\neq 0
$$
or/and 
$$
\HF_{m+1}^{(b,\,b+E)}\big(H^{(l)}\big)=\HF_{m+1}^{(b,\,b+E)}\big(H^{\# l}\big)\neq 0.
$$
In the former case, there exists an $l$-periodic orbit $\by$ with
action in the range $(a,\,a+E)$ and $|m-\Delta_{H^{(l)}}(\by)|\leq
n$. In the latter case, there exists an $l$-periodic orbit $\by$ with
action in the range $(b,\,b+E)$ and $|m+1-\Delta_{H^{(l)}}(\by)|\leq
n$.  Then
\begin{equation}
\labell{eq:lu-action}
0<\min \{\alpha-E,\beta\}< |\CA_{H^{(l)}}(\by)-\CA_{H^{(l)}}(\bx^l)|
\leq \max \{\alpha,\beta+E\}
\end{equation}
and
$$
|\Delta_{H^{(l)}}(\bx^{l})-\Delta_{H^{(l)}}(\by)|\leq 2n+1.
$$

Since the action gap is positive, $\by\neq \bx^l$. Furthermore, when
(i) holds, $\by$ is not a non-trivial recapping $\bx^l$ by
\eqref{eq:alpha+beta} and \eqref{eq:lu-action}. If condition (ii) is
satisfied, $\by$ is not a non-trivial recapping $\bx^l$ since the mean
index gap does not exceed $2n+1<2N$.  (Here, as in Case 1, it would be
sufficient to assume that $N>n$.) Thus $y\neq x^l$. We set $\nu_i=l$
and $\by_i=\by$.

This completes the proof of Theorem \ref{thm:gap1}.

\begin{Remark}
  It is easy to infer from the proof of the theorem that the upper
  bound on the action gap is in fact $E_0=\| H\|$ when $\alpha$ and
  $\beta$ satisfy \eqref{eq:alpha+beta}, for $\beta>0$ and
  $\alpha-E_0>0$ can be made arbitrarily small, and the upper bound on
  the mean index gap is $4n$.
\end{Remark}

\subsection{The action-index gap and the Conley conjecture}
\label{sec:gap-conley}
In this section, we will briefly discuss the relation between Theorem
\ref{thm:gap1} and the Conley conjecture.

As is pointed out in Section \ref{sec:gap}, the hypotheses of Theorem
\ref{thm:gap1} are automatically satisfied for any Hamiltonian $H$
with isolated periodic orbits when $M$ is symplectically aspherical.
In this case the theorem was proved in \cite{GG:gap} in a stronger
form. Namely, the sequence $\nu_i$ can be chosen to be a subsequence
of any quasi-arithmetic sequence of admissible iterations and the
action--index gap \eqref{eq:gap} is non-zero. With these two
amendments -- the second one is particularly important -- the theorem
readily implies the Conley conjecture for symplectically aspherical
manifolds; see~\cite{GG:gap}.

However, to the best of our understanding, once $N<\infty$ and/or
$\lambda_0<\infty$, Theorem \ref{thm:gap1}, even if it were proved in
this stronger form, would not directly imply the Conley conjecture in
the cases where it is expected to hold, e.g., for manifolds with large
$N$. For instance, we have been unable to infer Theorem
\ref{thm:conley} from Theorem \ref{thm:gap1}. Furthermore, even though
the assertion of Theorem \ref{thm:gap2} is at least as strong as the
aspherical counterpart of Theorem \ref{thm:gap1}, the Conley
conjecture fails in general under the hypotheses of Theorem
\ref{thm:gap2} as becomes clear by comparing Examples \ref{exam:
torus-action} and \ref{ex:gap2}.  Yet, Theorems \ref{thm:gap1} and
\ref{thm:gap2} yield a criterion for the existence of infinitely many
periodic orbits; see Corollary \ref{cor:augm-action}.

\section{Periodic orbits in the presence of a symplectically degenerate
  maximum}
\labell{sec:p-o-sdm}

The main objective of this section is to prove Theorem
\ref{thm:sdm}. The proof follows the same line of reasoning as the
argument from \cite{Gi:conley} where the theorem is established for
symplectically aspherical manifolds. The idea of the proof is the
squeezing trick (see \cite{BPS,GG:duke,Ke:sq}): the Hamiltonian $H$ is
bounded from above and below by functions $H_\pm$ such that the
monotone homotopy map for the iterated Hamiltonians $H^{(k)}_\pm$ is
non-zero for the action interval in question. Since this map factors
through the Floer homology of $H^{(k)}$, the latter Floer homology
group is also non-trivial.

The key to realizing this idea is a geometrical characterization of
symplectically non-degenerate maxima, building on the results of
\cite{Gi:conley,GG:gap,Hi}. Namely, we show that a Hamiltonian
diffeomorphism $\varphi_H$ with a symplectically degenerate maximum
$\bx$ can be generated by a Hamiltonian $K$ with a strict local
maximum at $p=x(0)$, for all times, and with arbitrarily small Hessian
at $p$; see Proposition \ref{def:sympl-deg}. Moreover,
$\varphi_H=\varphi_K$ in the universal covering of the group of
Hamiltonian diffeomorphisms, and hence $H$ and $K$ have the same
filtered Floer homology. The problem is thus reduced to proving the
theorem for a Hamiltonian $H=K$ with a genuine maximum at
$\bx=p$. Then the construction of $H_\pm$ and the analysis of the
monotone homotopy map can be carried out by fairly standard methods.

The main difference between the proof of Theorem \ref{thm:sdm} given
below and the argument for symplectically aspherical manifolds is that
in the latter case the Floer homology in question and the monotone
homotopy map can be determined explicitly (see
\cite{Gi:conley,GG:duke}) while in the setting of Theorem
\ref{thm:sdm} such a calculation looks problematic. To circumvent this
difficulty, we isolate a direct summand in the Floer homology groups
of $H_\pm$ which behaves exactly as if the manifold $M$ were
symplectically aspherical, e.g., $\R^{2n}$. As a consequence, the
monotone homotopy map between the corresponding summands for $H_\pm$
is non-zero. This procedure is of rather general nature and
independent interest and is described in detail in the next section.

\subsection{Direct sum decomposition of filtered Floer homology}
\label{sec:direct-sum}
Throughout this section, we assume that $M$ is weakly monotone and
geometrically bounded and rational; see \cite{AL,CGK} for the
definition and a detailed discussion of geometrically bounded
manifolds. Let $U\subset W\subset M$ be open sets with smooth boundary
and compact closure. For the sake of simplicity, we will assume that
the closed domains $\bar{U}$ and $\bar{W}$ are isotopic in $M$, e.g.,
$W$ is a small neighborhood of $\bar{U}$. Consider a Hamiltonian $F$,
which is constant on $M\ssminus U$, and set $C:=F\mid_{M\ssminus
U}$. Let $(a,\,b)$ be an interval such that
\begin{equation}
\label{eq:interval}
a\text{ and }b \text{ are outside }\CS(F)\text{ and }
(a,\,b)\cap (C+\lambda_0\Z)=\emptyset,
\end{equation}
where $\lambda_0$ is the rationality constant. (The set
$C+\lambda_0\Z$ is formed by action values of $F$ on trivial orbits
outside $U$ with arbitrary cappings.)  For instance, if
$\omega\mid_{\pi_2(M)}=0$, the second part of \eqref{eq:interval}
simply asserts that $C\not\in (a,\,b)$.  Note that under these
conditions the Floer homology $\HF^{(a,\,b)}_*(F)$ is defined (see
Remark \ref{rmk:floer-gb}) and that \eqref{eq:interval} implies, in
particular, that $b-a<\lambda_0$.

Assume now that all one-periodic orbits of $F$ with action in
$(a,\,b)$ are non-degenerate.  Denote by $\CF^{(a,\,b)}_*(F;U)$ the
vector space (over $\F$) generated by such orbits with cappings
contained in $U$ or, more precisely, with cappings equivalent to those
contained in $U$. (Since $W\supset U$ is isotopic to $U$, we can
equivalently require the cappings to be contained in $W$.) In
particular, the orbits in question must also be contained in $U$. Let
$\CF^{(a,\,b)}_*(F;M,U)$ be the vector space generated by the
remaining capped one-periodic orbits with action in the interval
$(a,\,b)$. (These orbits, but not their cappings, are also contained
in $U$ by \eqref{eq:interval}.) Thus, we have the direct sum
decomposition
\begin{equation}
\label{eq:direct-sum}
\CF^{(a,\,b)}_*(F)=\CF^{(a,\,b)}_*(F;U)\oplus \CF^{(a,\,b)}_*(F;M,U)
\end{equation}
on the level of vector spaces. Fix now an arbitrary, not necessarily
$F$-regular, almost complex structure $J_0$ compatible with $\omega$.

\begin{Lemma}
\label{lemma:direct-sum}
There exists a constant $\eps(U,W,M)>0$, independent of $F$, such that
whenever $b-a<\eps(U,W,M)$ the decomposition \eqref{eq:direct-sum} is
a direct sum of complexes for any $F$-regular almost complex structure
$J$ sufficiently close to $J_0$ and compatible with $\omega$.  Thus,
we have, in the obvious notation, the decomposition of Floer homology
\begin{equation}
\label{eq:direct-sum-h}
\HF^{(a,\,b)}_*(F)=\HF^{(a,\,b)}_*(F;U)\oplus \HF^{(a,\,b)}_*(F;M,U).
\end{equation}
The direct sum decompositions \eqref{eq:direct-sum} and
\eqref{eq:direct-sum-h} are preserved by the long exact sequence maps,
provided that all action intervals are shorter than $\eps(U,W,M)$, and
by monotone decreasing homotopy maps, as long as the Hamiltonians are
constant on $M\ssminus U$.
\end{Lemma}

Note that the constant $\eps(U,W,M)$ also depends on $J_0$ even though
this dependence is suppressed in the notation. However, $\eps(U,W,M)$
is bounded away from zero when $J_0$ varies within a compact
set. Furthermore, as is clear from the proof of the lemma, the
constant $\eps(U,W,M)$ is completely determined by the restriction of
$J_0$ to $\bar{W}\ssminus U$.

The proof of Lemma \ref{lemma:direct-sum} is essentially the standard
argument that a Floer connecting trajectory with small energy
connecting orbits from $\CF^{(a,\,b)}_*(F;U)$ must lie in $W$, based
on the Gromov compactness theorem. Postponing the proof, let us show
how the lemma extends, essentially by continuity, to degenerate
Hamiltonians; cf.\ Remark \ref{rmk:floer-gb}.  Observe first that, as
readily follows from the last assertion of the lemma, the
decomposition \eqref{eq:direct-sum-h} is independent of $J$ and, in
fact, of $J_0$ when $J_0$ varies within a fixed compact set and $b-a$
is sufficiently small.

Let $F$ be an arbitrary Hamiltonian which is constant on $M\ssminus U$
and meets the requirement of \eqref{eq:interval}. Due to the latter
condition, non-degeneracy is a generic requirement on the
Hamiltonian. More precisely, $F$ can be $C^2$-approximated by
Hamiltonians $\tF$ satisfying \eqref{eq:interval} and such that all
one-periodic orbits of $\tF$ with action in $(a,\,b)$ are
non-degenerate.  Fix $J_0$ and set, assuming that $b-a<\eps(U,W,M)$,
$$
\HF^{(a,\,b)}_*(F;U)=\HF^{(a,\,b)}_*(\tF;U)
$$
and
$$
\HF^{(a,\,b)}_*(F;M,U)=\HF^{(a,\,b)}_*(\tF;M,U).
$$
Clearly, these groups are independent of $\tF$ and $J$, when $\tF$ is
$C^2$-close to $F$ and $J$ is close to $J_0$.  As an immediate
consequence of Lemma \ref{lemma:direct-sum}, we obtain

\begin{Proposition}
\label{prop:direct-sum}
There exists a constant $\eps(U,W,M)>0$, independent of $F$, such that
the direct sum decomposition \eqref{eq:direct-sum} holds and is
preserved by the long exact sequence and decreasing monotone homotopy
maps, as long as all action intervals are shorter than $\eps(U,W,M)$
and the Hamiltonians are constant on $M\ssminus U$.
\end{Proposition}

\begin{Remark}
It is essential that in the last assertion of the proposition and of
the lemma where a monotone decreasing homotopy $F^s$ is considered,
the constant value $C_s=F^s\mid_{M\ssminus U}$ may depend on $s$ and
condition \eqref{eq:interval} is imposed only on the Hamiltonians
$F^0$ and $F^1$, but \emph{not} on all Hamiltonians $F^s$.  In
particular, the set $C_s+\lambda_0\Z$ is \emph{not} required to have
trivial intersection with the interval $(a,\,b)$ for all $s\in
[0,\,1]$, but only for $s=0$ and $s=1$.
\end{Remark}

\begin{proof}[Proof of Lemma \ref{lemma:direct-sum}]

The set $Z=\bar{W}\ssminus U$ is a ``shell'' bounded by
$S_0=\p\bar{U}$ and $S_1=\p\bar{W}$. Let $J$ be an almost complex
structure sufficiently close to $J_0$. Consider $J$-holomorphic curves
$v\colon\Sigma\to Z$ with boundary on $S_0\cup S_1=\p Z$ such that
$v(\Sigma)\cap S_0$ and $v(\Sigma)\cap S_1$ are both non-empty. By the
Gromov compactness theorem, there exists a constant $\eps(U,W,M)>0$
such that
\begin{equation}
    \label{eq:eq:lower-bound}
\int_v\omega> \eps(U,W,M)
\end{equation}
for all $v$ and all $J$ close to $J_0$.

Let $\bx$ and $\by$ be two capped one-periodic orbits of $F$ with
action in $(a,\,b)$, connected by a Floer anti-gradient trajectory
$u$, and such that one of these orbits, say $\bx$, has capping in
$U$. To prove that \eqref{eq:direct-sum} is a direct sum of
complexes, we need to show that the capping of the second orbit
$\by$ is also in $U$, provided that $b-a<\eps(U,W,M)$. To this end, it
suffices to prove that $u$ is contained in $W$, for then the capping
of $\by$ is homotopic to one contained in $U$.

The difference of actions $|\CA_F(\bx)-\CA_F(\by)|$ is equal to the energy
$$
E(u)=\int_{-\infty}^\infty\int_{S^1}\left\|\p_s u\right\|^2\,dt\,ds.
$$
Note that the part of $u$ that is not contained in $U$ is a
holomorphic curve in $M$ with boundary on $S_0=\p\bar{U}$. Let
$\Sigma\subset \R\times S^1$ be the collection of all $(s,t)$ such
that $u(s,t)\in Z$. Without loss of loss of generality, we may assume
that $u$ is transverse to $S_0$ and $S_1$, by perturbing these
hypersurfaces slightly, and thus $\Sigma$ is a closed domain in
$\R\times S^1$ with smooth boundary. Clearly, $v:=u|_\Sigma$ is a
holomorphic curve with values in $Z$ and boundary in $\p Z$. Thus,
\begin{equation}
    \label{eq:lower-bound}
\eps(U)> b-a\geq E(u)\geq \int_\Sigma \left\|\p_s u\right\|^2\,dt\,ds
=\int_\Sigma v.
\end{equation}
Arguing by contradiction, assume that $u$ is not entirely contained in
$W$. Then $v(\Sigma)\cap S_0$ and $v(\Sigma)\cap S_1$ are
non-empty. Hence, the upper bound \eqref{eq:eq:lower-bound} holds,
which is impossible by \eqref{eq:lower-bound}.

The fact that the decomposition \eqref{eq:direct-sum-h} is preserved
by the exact sequence maps is a formal consequence of the definitions
and of the complex decomposition \eqref{eq:direct-sum}.

The proof of the assertion that \eqref{eq:direct-sum-h} is preserved
by the monotone homotopy maps is identical to the above argument. (The
only modification is that, when $u$ is a homotopy connecting
trajectory, we have the inequality $|\CA_F(\bx)-\CA_F(\by)|\geq E(u)$
rather than an equality.)
\end{proof}

Furthermore, the proof of Lemma \ref{lemma:direct-sum} yields

\begin{Proposition}
\label{prop:direct-sum-independence}
The Floer homology $\HF^{(a,\,b)}_*(F;U)$ and the monotone homotopy
maps are independent of the ambient manifold $M$ when $b-a$ is
sufficiently small.
\end{Proposition}

More precisely, this proposition should be read as follows. Let a
closed symplectic manifold $\bar{W}$ with boundary be embedded as a
closed domain into weakly monotone, rational, geometrically bounded
symplectic manifolds $M_1$ and $M_2$. Assume that $F$ is a Hamiltonian
on $\bar{W}$, constant on $\bar{W}\ssminus U$. Thus, $F$ extends
naturally to $M_1$ and to $M_2$ as a function constant outside $U$.
Furthermore, assume that \eqref{eq:interval} is satisfied for both
$M_1$ and $M_2$, and $b-a<\min\{\eps(U,W,M_1),\eps(U,W,M_2)\}$.  Then,
there is a canonical isomorphism between the groups
$\HF^{(a,\,b)}_*(F;U)$ for the two ambient manifolds $M_1$ and
$M_2$. Furthermore, $\eps(U,W,M_1)=\eps(U,W,M_2)$ when the complex
structure on $M_1$ and the complex structure on $M_2$ coincide on
$\bar{W}\ssminus U$.

\begin{Example}
\label{ex:sa}
Assume that $M$ is symplectically aspherical. Then, regardless of the
length of the interval $(a,\,b)$, by the very definition of
$\CF^{(a,\,b)}_*(F;U)$ we have
$\CF^{(a,\,b)}_*(F)=\CF^{(a,\,b)}_*(F;U)$ and
$\HF^{(a,\,b)}_*(F)=\HF^{(a,\,b)}_*(F;U)$ if \eqref{eq:interval} is
satisfied. However, this homology group depends on $M$, unless $b-a$ 
is small.
\end{Example}

\begin{Example}
\label{ex:independence}
Assume that $\bar{W}$ is a Darboux ball embedded in $M$ and $F$ is as
above.  Then, combining Example \ref{ex:sa} with Proposition
\ref{prop:direct-sum-independence}, we conclude that
$\HF^{(a,\,b)}_*(F;U)$, when $b-a$ is sufficiently small, is equal to
the filtered Floer homology of $F$ viewed as a Hamiltonian on
$\R^{2n}$. In particular, when $M$ itself is symplectically
aspherical, we obtain equality between the filtered Floer homology of
$F$ on $M$ and $\R^{2n}$ for small action intervals.  (This
observation provides a shortcut to the calculation of the Floer
homology in \cite[Section 7]{Gi:conley} reducing the problem to the
case of functions on $\R^{2n}$, analyzed in \cite{GG:duke}.) In the
proof of Theorem \ref{thm:sdm}, Propositions \ref{prop:direct-sum} and
\ref{prop:direct-sum-independence} are utilized in a similar fashion.

\end{Example}

\begin{Remark}
    The proof of Lemma \ref{lemma:direct-sum} also shows that
    $\CF^{(a,\,b)}_*(F;U)$ equipped with the usual Floer differential
    is a complex with well-defined homology even when $M$ is not
    geometrically bounded; cf.\ \cite{GG:CMH}.
\end{Remark}

\subsection{Geometrical characterization of symplectically degenerate maxima}
\label{sec:geom-sdm}
In this section, we give a geometrical characterization of
symplectically degenerate maxima, following \cite{Gi:conley,GG:gap}.

Recall that the norm $\|v\|_\Xi$ of a tensor $v$ on a
finite-dimensional vector space $V$ with respect to a basis $\Xi$ is,
by definition, the norm of $v$ with respect to the Euclidean inner
product for which $\Xi$ is an orthonormal basis. For instance, let
$A\colon V\to V$ be a linear map with all eigenvalues equal to zero.
Then $\|A\|_{\Xi}$ can be made arbitrarily small by suitably choosing
$\Xi$. In other words, for any $\sigma>0$, there exists $\Xi$ such
that $\|A\|_{\Xi}<\sigma$. (Indeed, in some basis, $A$ is given by an
upper triangular matrix with zeros on the diagonal; $\Xi$ is then
obtained by appropriately scaling the elements of this basis.)  When
$\xi$ is a coordinate system near $p\in M$, we denote by $\xi_p$ the
natural coordinate basis in $T_pM$ arising from $\xi$.

\begin{Proposition}
\labell{def:sympl-deg}

Assume that a capped one-periodic orbit $\bx$ of $H$ is a
symplectically degenerate maximum. Fix a point $p\in M$, e.g.,
$p=x(0)$.  Then there exists a sequence of contractible loops $\eta_i$
of Hamiltonian diffeomorphisms generated by Hamiltonians $G^i$ such
that $\Phi_{G^i}(p)=\bx$, i.e., $\eta_i$ sends $p$ with trivial
capping to $\bx$, and the following conditions are satisfied:

\begin{itemize}
\item[(K1)] the point $p$ is a strict local maximum of $K_t^i$ for all
$t\in S^1$ and all $i$, where $K^i$ is given by  $H=G^i\# K^i$;
\item[(K2)]
there exist symplectic bases $\Xi^i$ in $T_pM$ such that
$$
\|d^2 (K_t^i)_p\|_{\Xi^i}\to 0\text{ uniformly in $t\in S^1$;}
$$
\item[(K3)] the linearization of the loop $\eta_i^{-1}\circ \eta_j$ at
$p$ is the identity map for all $i$ and $j$ (i.e.,
$d\big((\eta_i^t)^{-1}\circ \eta_j^t\big)_p=I$ for all $t\in S^1$)
and, moreover, the loop $(\eta_i^t)^{-1}\circ \eta_j^t$ is
contractible to $\id$ in the class of loops fixing $p$ and having the
identity linearization at $p$.
\end{itemize}
\end{Proposition}

\begin{Remark}
As is easy to see, this proposition gives, in fact, a necessary and
sufficient condition for $\bx$ to be a symplectically degenerate
maximum. Moreover, requiring (K1) and (K2) is already sufficient, and
thus (K3) is a formal consequence of these two conditions;
\cite{GG:gap}.
\end{Remark}

\begin{proof}
By Proposition \ref{prop:loop1}, we may assume without loss of
generality that $\bx$ is already a constant orbit $p$ of $H$ with
trivial capping. Then, the existence of local loops $\eta_i$ generated
by Hamiltonians $G_i$ defined on neighborhoods of $p$ is established
in \cite[Section 5]{Gi:conley}.  These loops have zero Maslov
index. To see this, first note that by (K3) all loops $\eta_i$ have
the same Maslov index. Furthermore, by (MI7),
$$
2\mu_{G^i}(p)+\Delta_{K^i}(p)=\Delta_H(p)=0.
$$
Here, $\mu_{G^i}(p)\in\Z$ and $\Delta_{K^i}(p)\to 0$ as $i\to\infty$
by (K2).  Thus, $\mu_{G^i}(p)=0$. It follows now from Proposition
\ref{prop:loop2} and Remark \ref{rmk:loop-ext} that the local loops
$\eta_i$ extend to global loops with required properties.
\end{proof}

\subsection{Proof of Theorem \ref{thm:sdm}}
\label{sec:prf-sdm}
By Proposition \ref{def:sympl-deg}, it is sufficient to prove the
theorem for the Hamiltonian $K^1$ in place of $H$ and the constant
orbit $p$ of $K^1$ with trivial capping in place of $\bx$.  Thus,
keeping the notation $H$ for the Hamiltonian $K^1$, we may assume
throughout the proof that
\begin{itemize}
\item the point $p$ is a strict local maximum of $H_t$ for all
$t\in S^1$ and
\item $d(\eta_i^t)_p=I$ for all $t\in S^1$.
\end{itemize}
Furthermore, without loss of generality, we may assume that $H\geq 0$
and $c=H_t(p)$ is independent of $t$.

Our goal is to construct functions $H_\pm$ such that $H_+\geq H\geq
H_-$ and the monotone homotopy map
\[
\Psi\colon  \HF_{2n+1}^{(kc+\delta_k,\,kc+\eps)}\big(H_+^{(k)}\big)\to
    \HF_{2n+1}^{(kc+\delta_k,\,kc+\eps)}\big(H_-^{(k)}\big)
\]
is non-zero, provided that $k$ is large enough; see Fig.\ 1. Then,
since this map factors through
$\HF_{2n+1}^{(kc+\delta_k,\,kc+\eps)}\big(H^{(k)}\big)$, the latter
group is also non-trivial.
\begin{figure}[htbp]
\begin{center}
\input{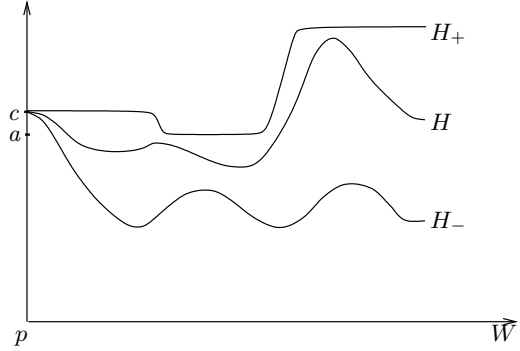}
\caption{The functions $H$ and $H_\pm$}
\end{center}
\end{figure}

\subsubsection{The functions $H_\pm$.}
\label{sec:H+}

Fix a system of canonical coordinates on small Darboux balls $U\Subset
W$ centered at $p$ and let the function $\rho$ on $W$ be the square of
the distance to $p$ with respect to this coordinate system. The
neighborhood $W$ is required to be so small that $c=H(p)$ is a strict
global maximum of $H$ on $W$ and $W$ is contained in a larger Darboux
ball.  Fix also a compatible with $\omega$ almost complex structure
$J_0$ on $M$ coinciding with the standard complex structure on a
neighborhood of $\bar{W}$. In this case, as was pointed out in Section
\ref{sec:direct-sum}, $\eps(U,W,M)=\eps(U,W,\R^{2n})$.  The parameter
$\eps>0$ from the assertion of the theorem is required to satisfy
\begin{equation}
\label{eq:eps}
\eps<\eps(U,W,M)=\eps(U,W,\R^{2n}).
\end{equation}

Pick four balls centered at $p$ in $U$:
$$
B_{r_-}\subset B_{r}\subset B_R \subset B_{R_+} \Subset U.
$$

Let $H_+$ be a function of $\rho$, also treated as a function on $U$,
with the following properties (see Fig.\ 1):

\begin{itemize}

\item $H_+\geq H$,

\item $H_+|_{B_{r_-}}\equiv c=H(p)$,

\item on the shell $B_{r}\ssminus B_{r_-}$ the function $H_+$ is
monotone decreasing,

\item $H_+\equiv a$ on the shell $B_R\ssminus B_{r}$,

\item $H_+$ is monotone increasing on the shell $B_{R_+}\ssminus B_R$,

\item on $U\ssminus B_{R_+}$, the function $H_+$ is constant and equal to
      $\max H_+$.

\end{itemize}

Furthermore, we extend $H_+$ to $M$ by setting it to be constant and
equal to $\max H_+$ on the complement of $U$. The constant $\max H_+$
is chosen so that $H\leq H_+$ on $M$ and $\max H_+>c+\eps$.

The precise description of $H_+$ and the choice of the parameters in
its construction are not explicitly used in this proof. It is
essential though that the function $H_+$ is defined exactly as in
\cite[Section 7]{Gi:conley}.  At this stage, let us fix $H_+$.

Assume that $k$, depending on $H_+$, is sufficiently large and
$\delta_k>0$, depending on $H_+$ and $k$, is sufficiently small. Fix
$k$ and $\delta_k$.

Then the Hamiltonian $H_-\geq 0$, depending on $H^+$ and $k$, is
constructed as follows.  Recall that there exist
\begin{itemize}
\item a loop $\eta^t=\eta^t_i$, $t\in S^1$, of Hamiltonian
diffeomorphisms fixing $p$ and
\item a system of canonical coordinates $\xi=\xi^i$ on a small
neighborhood $V\subset U$ of $p$
\end{itemize}
such that the Hamiltonian $K=K^i$ generating the flow
$(\eta^t)^{-1}\circ\varphi^t_H$ has strict local maximum at $p$ and
$\max_t\|d^2 (K_t)_p\|_{\xi_p}\to 0$ as $i\to\infty$.  Moreover, the
loop $\eta$ has identity linearization at $p$, i.e., $d(\eta^t)_p=I$
for all $t\in S^1$, and is contractible to $\id$ in the class of such
loops. Let $\eta_s$ be a homotopy from $\eta$ to $\id$ such that
$d(\eta^t_s)_p\equiv I$ and let $G_t^s$ be the one-periodic
Hamiltonian generating $\eta^t_s$ and normalized by $G_t^s(p)\equiv
0$.  The condition $d(\eta^t_s)_p=I$ is equivalent to
$d^2(G_t^s)_p=0$.  We normalize $K$ by requiring that $K_t(p)\equiv c$
or, equivalently, by $H=G\# K$.

When $i$ is sufficiently large and hence $k\max_t\|d^2
(K_t)_p\|_{\xi_p}$ is sufficiently small, there exists a ``bump
function'' $F$ (with respect to the coordinate system $\xi$) on a
neighborhood $V\subset B_{r_-}$ of $p$ such that

\begin{itemize}
    \item $k\|d^2 F_p\|_{\xi_p}$ is sufficiently small,
    \item $F\leq K$ and $F(p)=c=K(p)$ is the absolute maximum of $F$,
    \item $F^s := G^s\# F\leq H_+$ for all $s$.
\end{itemize}
Note that here we do not require $F$ to be supported in $V$, but only
to be constant and equal to $\min F$ outside $V$.  By Example
\ref{ex:isospec}, $F^s$ is an isospectral homotopy beginning with
$$
H_-:=G^0\#F\leq G^0\#K=H
$$
and ending with $F^1=F$. Throughout this homotopy, $F^s\leq H_+$.

Furthermore, the homotopy $(F^s)^{(k)}$ from $H_-^{(k)}$ to $F^{(k)}$
is also isospectral and $(F^s)^{(k)}\leq H_+^{(k)}$ for all $s$. (Note
in this connection that in general an iteration of an isospectral
homotopy need not be isospectral. However, an iteration of a homotopy
of the form $G^s\# K$, where all $G^s$ generate loops of Hamiltonian
diffeomorphisms and are normalized, is automatically isospectral.)

\subsubsection{The Floer homology of $H_\pm$ and the monotone homotopy map}
When the manifold $M$ is symplectically aspherical, the Floer homology
groups of the functions $H_+^{(k)}\geq F^{(k)}$, for the intervals in
question, and the monotone homotopy map are determined in
\cite{Gi:conley}. In fact, it suffices to carry out the calculation
for $M=\R^{2n}$, which is done in \cite{GG:duke}, for then, by Example
\ref{ex:independence}, the result holds for any symplectically
aspherical manifold whenever $\eps$ is small. When $M=\R^{2n}$ and
$\eps<\eps(U,W,\R^{2n})$, it is shown that the map
$$
\F\cong\HF_{2n+1}^{(kc+\delta_k,\,kc+\eps)}\big(H_+^{(k)}\big)
\stackrel{\cong}{\to}
\HF_{2n+1}^{(kc+\delta_k,\,kc+\eps)}\big(F^{(k)}\big)\cong\F
$$
is an isomorphism. (At this point we use the requirement that $k\|d^2
F_p\|_{\xi_p}$ is sufficiently small.)

Thus, returning to the setting of Theorem \ref{thm:sdm},
we have the isomorphism
$$
\F\cong\HF_{2n+1}^{(kc+\delta_k,\,kc+\eps)}\big(H_+^{(k)};U\big)
\stackrel{\cong}{\to}
\HF_{2n+1}^{(kc+\delta_k,\,kc+\eps)}\big(F^{(k)};U\big)\cong\F,
$$
by \eqref{eq:eps} and Proposition
\ref{prop:direct-sum-independence}. Since $\eps<\eps(U,W,M)$ and hence
these groups enter the filtered Floer homology of $H_+^{(k)}$ and
$F^{(k)}$ as direct summands, we conclude that the map
$$
\HF_{2n+1}^{(kc+\delta_k,\,kc+\eps)}\big(H_+^{(k)}\big)
\stackrel{\neq 0}{\longrightarrow}
\HF_{2n+1}^{(kc+\delta_k,\,kc+\eps)}\big(F^{(k)}\big)
$$
is non-zero.

Consider the diagram
$$
\xymatrix{
{\HF_{2n+1}^{(kc+\eps,\,kc+\delta_k)}\big(H_+^{(k)}\big)}
 \ar[d]_\Psi\ar[rd]^{\neq 0} &\\
{\HF_{2n+1}^{(kc+\eps,\,kc+\delta_k)}\big(H_-^{(k)}\big)} \ar[r]^{\cong} &
{\HF_{2n+1}^{(kc+\eps,\,kc+\delta_k)}\big(F^{(k)}\big)}
}
$$
where the horizontal isomorphism is induced by the isospectral
homotopy $(F^s)^{(k)}$ from $H_-^{(k)}$ to $F^{(k)}$ and the remaining
two arrows are monotone homotopy maps. Since $(F^s)^{(k)}\leq
H_+^{(k)}$ for all $s$, the diagram commutes; see Section
\ref{sec:hom} and, in particular, Example \ref{ex:isospec} and the
diagram \eqref{eq:diag-iso}.  As we have just shown, the diagonal
arrow is a non-zero map, and hence so is the map $\Psi$. This
concludes the proof of the theorem.

\begin{Remark}
  Note that this argument, reducing the calculation to the case of
  $M=\R^{2n}$, also provides a simple ending of the original proof
  from \cite{Gi:conley}.
\end{Remark}

\section{The Ljusternik--Schnirelman theory}
\label{sec:LS}
Our goal of this section is to prove Theorems \ref{thm:gap2} and
\ref{thm:CPn}.  The proofs differs significantly from the proofs of
other theorems in this paper and rely on the machinery of action
selectors and the Hamiltonian Ljusternik--Schnirelman theory.

\subsection{The classical Ljusternik--Schnirelman theory: critical value selectors}
To set the stage for an overview of the theory of Hamiltonian action
selectors and for the proof of Theorem \ref{thm:gap2}, let us briefly
recall the definition and basic properties of classical critical value
selectors in the Ljusternik--Schnirelman theory.

Let $M$ be a closed, $m$-dimensional manifold and $f\in
C^\infty(M)$. To $u\in H_*(M)$, we associate the \emph{critical value
selector} by the formula
$$
\sls_u(f)= \inf\{ a\mid u\in \im(i^a)\}=\inf\{ a\mid j^a(u)=0\},
$$
where $i^a\colon H_*(\{f\leq a\})\to H_*(M)$ and $j^a\colon H_*(M)\to
H_*(M,\{f\leq a\})$ are the natural ``inclusion'' and ``quotient''
maps. By definition, $\sls_0(f)=-\infty$.  As is well known,
$\sls_u(f)$ has the following properties:
\begin{itemize}

\item $\sls_u(\const)=\const$ (normalization) and 
$$
\sls_1(f)=\min f\leq \sls_u (f)\leq max f=\sls_{[M]}(f);
$$
\item $\sls_u(f)$ is Lipschitz in $f$ with respect to the
$C^0$-topology (continuity);

\item $\sls_{u\cap w}(f+g)\leq \sls_u(f)+\sls_w(g)$ (the triangle inequality)
and $\sls_{\alpha u}(f)=\sls_{u}(f)$ for any non-zero $\alpha\in \F$;

\item $\sls_u(f)$ is a critical value of $f$ (criticality or the
minimax principle);

\item $\sls_{u\cap w}(f)\leq \sls_u(f)$ and, moreover, if $w\neq [M]$
and the critical points of $f$ are isolated, we have strict inequality
\begin{equation}
\label{eq:LS}
\sls_{u\cap w}(f)<\sls_u(f).
\end{equation}
\end{itemize}
The proofs of these well-known facts can be easily extracted from any
standard treatment of the Ljusternik--Schnirelman theory.

As an immediate consequence of \eqref{eq:LS}, we obtain the classical
result of Ljusternik and Schnirelman asserting that the number of
distinct critical values of $f$ is no smaller than $\cl(M)+1$,
whenever the critical points of $f$ are isolated; see, e.g., \cite{HZ}
and references therein. Here, $\cl(M)$ stands for the cup-length of
$M$ over $\F$, i.e., the maximal number $k$ of elements
$u_1,\ldots,u_k$ in $H_{*<m}(M)$ such that $u_1\cap\ldots \cap u_k\neq
0$. For this reason, we will refer to \eqref{eq:LS} as the
\emph{Ljusternik--Schnirelman inequality}.

\begin{Remark} 
\label{rmk:LS-ref}
The Ljusternik--Schnirelman inequality \eqref{eq:LS} can be refined as
follows.  Assume that $\sls_{u\cap w}(f)=\sls_u(f)$ and let $\Sigma$
be the set of critical points of $f$ on the level
$\{f=\sls_u(f)\}$. Denote by $w^*\in H^*(M)$ the Poincar\'e dual of
$w$. Then the restriction of $w^*$ to any neighborhood of $\Sigma$ is
non-zero. In particular, the image of $w^*$ in the Alexander--Spanier
cohomology of $\Sigma$ is also non-zero. When $w=[M]$ this implies
that $\Sigma\neq\emptyset$ (criticality). If $w\in H_{*<m}(M)$, we
infer that for the strict inequality \eqref{eq:LS} to fail, the
critical set $\Sigma$ must contain a cycle of degree $m-\deg w>0$
comprised of critical points, and hence $\Sigma$ must be infinite.
\end{Remark}

\subsection{The Hamiltonian Ljusternik--Schnirelman theory: action selectors}
The theory of Hamiltonian \emph{action selectors} or \emph{spectral
invariants}, as they are usually referred to, can be viewed as the
theory of critical value selectors applied to the action functional
$\CA_H$ in place of a smooth function $f$. A complete Floer--theoretic
version of such a theory for symplectically aspherical manifolds was
developed in \cite{schwarz} although the first versions of the theory
go back to \cite{HZ,Vi:gen}. (Applications of action selectors extend
far beyond the scope of our discussion here; see, e.g.,
\cite{FS,EP1,EP2,Gi:we,Gu,HZ,MS,schwarz,Vi:gen} and references
therein.) In this section, we mainly follow \cite{Oh}, where the
actions selectors are studied for Hamiltonians on closed, weakly
monotone manifolds.

\subsubsection{Basic properties of the action selectors}
Let $M$ be such a manifold and let $H$ be a Hamiltonian on $M$.  To a
non-zero element $u\in \HQ_*(M)\cong\HF_*(H)$, we associate the
\emph{action selector} $\s_u$ by a formula virtually identical to the
definition of~$\sls$:
$$
\s_u(H)= \inf\{ a\in \R\ssminus \CS(H)\mid u\in \im(i^a)\}
=\inf\{ a\in \R\ssminus \CS(H)\mid j^a(u)=0\},
$$
where now $i^a\colon \HF_*^{(-\infty,\,a)}(H)\to \HF_*(H)$ and
$j^a\colon \HF_*(H)\to\HF_*^{(a,\, \infty)}(H)$ are the natural
``inclusion'' and ``quotient'' maps; see Section \ref{sec:Floer}.
Then $\s_u(H)>-\infty$. (This is a non-trivial statement unless $M$ is
assumed to be rational; see \cite{Oh} for a proof.)  As in the
finite-dimensional case, by definition, $\s_0(H)=-\infty$. In what
follows, for the sake of convenience, we will always assume that $u$
is homogeneous. (This assumption is not really necessary: $u$ could be
a sum of homology classes of different degrees.)  When $H$ is
non-degenerate,
$$
\s_u(H)=\inf_{[\sigma]=u}\CA_H(\sigma),
$$
where we set $\CA_H(\sigma)=\max\{\CA_H(\bx)\mid \alpha_{\bx} \neq 0\}$
for $\sigma=\sum\alpha_{\bx} \bx\in\CF_*(H)$.  

The action selectors $\s_u$ have the following properties:

\begin{itemize}

\item[(AS1)] Normalization: $\s_u(H)=\sls_u(H)$ if $u\in H_*(M)$ and
  $H$ is autonomous and $C^2$-small.

\item[(AS2)] Continuity: $\s_u$ is Lipschitz in $H$ in the $C^0$-topology.

\item[(AS3)] Monotonicity: $\s_u(H)\geq \s_u(K)$, whenever $H\geq K$ pointwise.

\item[(AS4)] Hamiltonian shift: $\s_u(H+a(t))=\s_u(H)+\int_0^1a(t)\,dt$, where 
      $a\colon S^1\to\R$.

\item[(AS5)] Symplectic invariance:
$\s_{\varphi_*(u)}(H)=\s_u(\varphi^* H)$ for any symplectomorphism
$\varphi$, and hence $\s_u(H)=\s_u(\varphi^* H)$ if $\varphi$ is
isotopic to $\id$ in the group of symplectomorphisms.

\item[(AS6)] Homotopy invariance: $\s_u(H)=\s_u(K)$, when
$\varphi_H=\varphi_K$ in the universal covering of the group of
Hamiltonian diffeomorphisms, both $H$ and $K$ are normalized, and $M$
is rational.

\item[(AS7)] Quantum shift: $\s_{\alpha u}(H)=\s_u(H)+I_\omega(\alpha)$, 
             where $\alpha\in \Lambda$.

 \item[(AS8)] Valuation inequality:
             $\s_{u+w}(H)\leq\max\{\s_u(H),\,\s_w(H)\}$ and, moreover,
             the inequality is strict if $\s_u(H)\neq \s_w(H)$.

\item[(AS9)] Triangle inequality: $\s_{u*w}(H\# K)\leq\s_u(H)+\s_w(K)$.

\item[(AS10)] Spectrality: $\s_u(H)\in \CS(H)$, provided that $M$ is
rational or otherwise $H$ is non-degenerate. More specifically, there
exists a capped one-periodic orbit $\bx$ of $H$ such that
$\s_u(H)=\CA_H(\bx)$ and $n\leq |\Delta_H(\bx)-\deg u|\leq
2n$. Furthermore, $\HF_{\deg u}(H,\bx)\neq 0$ if $x$ is isolated. In
particular, $\MUCZ(\bx)+n=\deg u$ when $H$ is non-degenerate.

\end{itemize}

We refer the reader to \cite{Oh} for the proofs of (AS1)-(AS9); see
also \cite{En,EP2,schwarz,MS}.  Some of these properties (e.g., (AS7)
and (AS8)) follow immediately from the definitions, while some (e.g.,
(AS9)) are quite non-trivial to establish. The spectrality property
(AS10) is proved in \cite{Oh} in the rational case and in \cite{U1}
for non-degenerate Hamiltonians.

\subsubsection{The Hamiltonian Ljusternik--Schnirelman inequality}
The additional property of the action selector that is needed for the
proof of Theorem \ref{thm:gap2} is an analogue of the
Ljusternik--Schnirelman inequality \eqref{eq:LS}. Set
$$
\HQ_*^-(M)=H_{*<2n}(M)\otimes\Lambda.
$$
Thus, $\HQ_*(M)=\HQ_*^-(M)\oplus [M]\Lambda$, but, in general,
$\HQ_*^-(M)$ is not a subalgebra of $\HQ_*(M)$. 

\begin{Proposition}[Hamiltonian Ljusternik--Schnirelman inequality]
\label{prop:LSF}
For any $H$ and $u$ and $w$ in $\HQ_*(M)$ we have 
$\s_{u*w}(H)\leq \s_u(H)+I_\omega(w)$. 
Moreover,
\begin{equation}
\label{eq:LSF}
\s_{u*w}(H)<\s_u(H)+ I_\omega(w),
\end{equation}
provided that $w\in \HQ_*^-(M)$ and the periodic orbits of $H$ are
isolated.
\end{Proposition}

The authors are not aware of any reference to a proof of this
proposition although the result bears a similarity to arguments from,
e.g., \cite{Fl,Ho,LO,schwarz2}. (For instance, for rational manifolds,
the result appears to be implicitly contained in \cite{schwarz2} and
the symplectically aspherical case is treated in \cite{schwarz}.)
Since Proposition \ref{prop:LSF} is particularly important for what
follows, we outline its proof, which is reminiscent of the arguments
in \cite{Gu}.

\begin{proof}
  Observe that $\s_u(0)=I_\omega(u)$ as is clear from (AS1) and
  (AS7). Then the non-strict inequality follows immediately from the
  triangle inequality (AS9) with $K\equiv 0$.  The proof of the strict
  inequality \eqref{eq:LSF} is based on several preliminary
  observations.

Consider first an arbitrary class $w=\sum_A w_A e^A\in \HQ_*(M)$. We
claim that
\begin{equation}
\label{eq:LSF-p1}
\s_w(g)\leq \sup_A\{\sls_{w_A}(g)\}+ I_\omega(w)
\end{equation}
when $g$ is a sufficiently $C^2$-small function on $M$. 
To see this, note that, by (AS1) and (AS8),
$$
\s_w(g)\leq \sup_A\{\sls_{w_A}(g)+I_\omega(A)\},
$$
and \eqref{eq:LSF-p1} readily follows when $M$ is rational.  To prove
\eqref{eq:LSF-p1} in the general case, note that the set $\{
I_\omega(A)\mid w_A\neq 0\}$ is discrete in $\R$ regardless of whether
$M$ is rational or not. Let us order this set as
$\lambda_1>\lambda_2>\ldots$, where $I_\omega(w)=\lambda_1$.  All
critical values of $g$, and hence, in particular, $\sls_{w_A}(g)$, lie
in the range $[\min g,\,\max g]$. Thus, when $\max g-\min
g<(\lambda_1-\lambda_2)/2$, the right hand side of the latter
inequality does not exceed $\sup\{\sls_{w_A}(g)\}+\lambda_1$, which
proves \eqref{eq:LSF-p1}.

Fix now a finite collection of points $p_1,\ldots,p_m$ in $M$ and
assume that $w\in \HQ_*^-(M)$. Let $g\leq 0$ be a $C^2$-small function on
$M$ which is equal to zero on small neighborhoods of these points and
otherwise strictly negative. Then, clearly, $\sls_v(g)<-\delta_g<0$
for all $v\in H_{*<2n}(M)$ and some $\delta_g>0$ depending of
$g$. Thus, by \eqref{eq:LSF-p1}, we have
\begin{equation}
\label{eq:LSF-p2}
\s_w(g)\leq I_\omega(w)-\delta_g<I_\omega(w).
\end{equation}
Furthermore, let $a\colon S^1\to\R$ be a non-negative, $C^2$-small
function, equal to zero outside a small neighborhood of $0\in
S^1$. Set $f=ag$. The Hamiltonian flow of $f$ is a
reparametrization of the flow of $g$ through time $\eps=\int_0^1
a(t)\,dt$. Hence, when $g$ and $a$ are sufficiently $C^2$-small, $\eps
g$ and $f$ have the same periodic orbits (the critical points of $g$)
and the same action spectrum. Clearly, this is also true for every
function in the linear family $f_s=(1-s)\eps g+s f$, with $s\in
[0,\,1]$, connecting $\eps g$ and $f$. By the continuity of $\s_w$ (see (AS2)), we conclude
that $\s_w(\eps g)=\s_w(f)$ and \eqref{eq:LSF-p2} turns into
\begin{equation}
\label{eq:LSF-p3}
\s_w(f)\leq I_\omega(w)-\delta_{\eps g}<I_\omega(w).
\end{equation}

Now we are in a position to finish the proof. Let $p_1,\ldots,p_m$ be
the fixed points of $\varphi_H$. Consider the Hamiltonian $H\#
f$. Assume that the neighborhoods of the fixed points on which
$g\equiv 0$ are sufficiently small, the interval on which $a>0$ is
sufficiently short, and also $g$ and $a$ are $C^2$-small. Then it is
not hard to see that $H\# f$ has the same periodic orbits as $H$ and,
moreover, $\CS(H\# f)=\CS(H)$.  As above, the same holds for all
Hamiltonians $H\# (sf)$ with $s\in [0,\,1]$. Again, by the continuity
property (AS2), we have $\s_{u*w}(H)=\s_{u*w}(H\# f)$. Finally, using
the triangle inequality (AS9) and \eqref{eq:LSF-p3}, we obtain
$$
\s_{u*w}(H)=\s_{u*w}(H\# f)\leq \s_u(H)+\s_w(f)<\s_u(H)+I_\omega(w).
$$
\end{proof}

\subsubsection{Digression: the Arnold conjecture}
Proposition \ref{prop:LSF} provides a natural framework for a
discussion of the degenerate case of the Arnold conjecture asserting
that the number of one-periodic orbits of $H$ is bounded from below by
$\cl(M)+1$.  (The conjecture is still open for a general rational,
weakly monotone manifold $M$.) For instance, the proposition
immediately implies the degenerate case of the conjecture when $M$ is
symplectically aspherical, originally proved in \cite{F:c-l,Ho}; see
also \cite{HZ}.  Although in the presence of non-trivial cappings
Proposition \ref{prop:LSF} does not address one of the main
difficulties -- showing that the orbits found are geometrically
distinct -- it allows one to easily reprove some of the known results.
Here, for the sake of completeness, we illustrate this utilization of
Proposition \ref{prop:LSF} by several examples.

The general line of reasoning is absolutely standard. Namely, assume
that we have $m+1$ non-zero homology classes $[M]=u_0,u_1,\ldots, u_m$
in $\HQ_*(M)$ such that $u_{k}=u_{k-1}* w_{k}$ with $w_{k}\in
\HQ_*^-(M)$ and $I_\omega(w_k)\leq 0$. Then, by Proposition
\ref{prop:LSF}, $\s_{u_{k}}(H)<\s_{u_{k-1}}(H)$ for any Hamiltonian
$H$ on $M$ with isolated fixed points. (Every weakly monotone $M$
admits $\cl(M)+1$ such classes.) Thus, when the spectrality property
(AS10) holds, we have $m+1$ capped one-periodic orbits
$\bx_0,\ldots,\bx_m$ such that
\begin{equation}
\label{eq:AC}
\CA_H(\bx_{k})<\CA_H(\bx_{k-1})\text{ and } |\Delta_H(\bx_{k})-\deg
u_k|\leq 2n.
\end{equation}
In general, the orbits $x_k$ need not be distinct, and showing that at
least some of the orbits are requires additional assumptions on $M$
and presents the main difficulty in proving the Arnold conjecture even
when $M$ is rational and weakly monotone. Below, obtaining a lower
bound on the number of geometrically distinct one-periodic orbits, we
closely follow the original sources, with some simplifications
afforded by Proposition \ref{prop:LSF} and the usage of the mean
index, and restrict our attention only to the examples where this
lower bound is nearly obvious.

Among the results that readily follow from Proposition \ref{prop:LSF}
are, for instance:

\begin{itemize}
\item[(i)] The Arnold conjecture for rational manifolds, when
$\|H\|<\lambda_0$, \cite{schwarz2}.

\item[(ii)] The Arnold conjecture for negatively monotone manifolds
with $N\geq n$,~\cite{LO}.

\item[(iii)] The Arnold conjecture for $\CP^n$, \cite{Fo,FW,Fl}, or,
more generally, the existence of at least $n+1$ distinct orbits when
$M=\CP^n\times P$, where $P$ is symplectically aspherical,
\cite{Oh1,schwarz2}, and the existence of $\gcd(n_1+1,\ldots,n_l+1)$
distinct orbits on a monotone product $M=\CP^{n_1}\times\ldots
\times\CP^{n_l}$,~\cite{Fl}.

\item[(iv)] The existence of at least $2N/(2n-d)$ distinct
one-periodic orbits when $M$ is monotone with $N>0$ and there is a
non-nilpotent element $w\in \HQ_{d<2n}(M)$, \cite{schwarz2}.
\end{itemize}

\begin{proof}[A brief note on the proofs] In all four cases, the
  choice of the elements $u_k$ is the most straightforward. For
  instance, in (i) and in (ii), we can take $w_k\in H_{*<2n}(M)$ such
  that $w_1\cap\ldots\cap w_m=1$ and $m=\cl(M)$.  Then, in (i), the
  assertion that the orbits $x_k$ are geometrically distinct follows
  immediately from the well-known bound $0\leq\s_{[M]}(H)-\s_1(H)\leq
  \|H\|$; see, e.g., \cite{Oh,schwarz}. In (ii), this is a consequence
  of the assumption that $M$ is negative monotone and $N\geq n$, and
  of the bound $\Delta_H(\bx_j)-\Delta_H(\bx_i)\geq d_i-d_j-2n>-2n$,
  where $j>i$ and $d_k=\deg u_k$, which, in turn, follows from the
  second part of \eqref{eq:AC}.

  In (iii), with $M=\CP^n\times P$, we have $u_k=w^k\otimes[P]$, where
  $w$ is the generator of $H_{2n-2}(\CP^n)$, and $m=n$; cf.\ Example
  \ref{ex:cpn}. The fact that all orbits are geometrically distinct
  can be easily seen from, e.g., Lemma \ref{lemma:gap2}. The case of
  the monotone product of projective spaces is treated in a similar
  way, keeping track of both the action and the mean
  index. Alternatively, (iii) follows immediately from (iv).

  In (iv), the argument is slightly different; \cite{Fo,schwarz2}.
  Set $u_k=w^k$ with $k=0,1,\ldots$ and consider the infinite
  collection of the orbits $\bx_k$. Note that $\deg u_k=dk-2n(k-1)$,
  when $k\geq 1$.  Thus, by \eqref{eq:AC}, this collection contains
  $L/(2n-d)$ orbits with mean index in the range $[-L,\,0]$, up to an
  error independent of $L\gg 0$.  Assume that there are exactly $l$
  geometrically distinct orbits within the collection $\{\bx_k\}$.
  Each of these orbits contributes at most $L/2N$ capped orbits to the
  list of orbits with mean index in $[-L,\,0]$, again up to a bounded
  error. Thus, we can totally have at most $lL/2N$ orbits with action
  in $[-L,\,0]$. Dividing by $L\to\infty$, we conclude that $l\geq
  2N/(2n-d)$.
\end{proof}

\begin{Remark}
  All of the above arguments concern, in fact, the action or/and mean
  index spectrum of $H$.  For instance, in (i) and in (iii) with
  $M=\CP^n\times P$, we showed that $\CS(H)$ contains at least
  $\cl(M)+1$ distinct points in the interval $[\s_1(H),\,
  \s_{[M]}(H)]$ of length not exceeding $\lambda_0$; cf.\ Theorems
  \ref{thm:gap2} and \ref{thm:gap21}. The proof of (ii) is based on a
  lower bound on the mean index gap
  $\Delta_H(\bx_j)-\Delta_H(\bx_i)$. The assertion (iv) relies on an
  estimate of the asymptotic density of the mean index
  spectrum. Namely, we showed that the asymptotic density (defined in
  the obvious way) is bounded from below by $1/(2n-d)$. On the other
  hand, every periodic orbit generates, via recapping, a set of orbits
  with density $1/2N$. Hence, the number of distinct orbits is at
  least $2N/(2n-d)$.
\end{Remark}

\subsection{The \emph{a priori} bound on the action-index gap}
The goal of this section is to prove Theorem \ref{thm:gap2}. We
establish the following slightly more general result:

\begin{Theorem}
\labell{thm:gap21} 
Let $(M^{2n},\omega)$ be a closed, weakly monotone symplectic manifold
such that there exist $u\in \HQ_*^-(M)$ and $w\in \HQ_*^-(M)$ and an
element $\alpha\in\Lambda$ satisfying the homological condition of
Theorem \ref{thm:gap2}:
$$
[M]=(\alpha u)*w.
$$
Let $H$ be a Hamiltonian on $M$ with isolated one-periodic
orbits. Assume in addition that $I_\omega(u)\geq 0$ and one of the
following conditions holds:
\begin{itemize}

\item[(a)] $M$ is rational and
$I_\omega(\alpha)+I_\omega(w)\leq\lambda_0$;

\item[(b)] $0<|2n-\deg u|<2N$ and $H$ is non-degenerate.
\end{itemize}
Then the flow of $H$ has two geometrically distinct (capped)
one-periodic orbits $\bx$ and $\by$ such that
\begin{equation}
\label{eq:action}
0\leq -I_\omega(u)< \A_{H}(\bx)-\A_{H}(\by)<I_\omega(\alpha)+I_\omega(w)
\end{equation}
and
\begin{equation}
\label{eq:index}
\big|\Delta_{H}(\bx)-n\big|\leq n
\text{ and }
n\leq\big|\Delta_{H}(\by)-\deg u\big|\leq 2n.
\end{equation}
\end{Theorem}

\begin{Remark}
\label{rmk:gap21}
When the Hamiltonian $H$ is non-degenerate, the mean index bounds
\eqref{eq:index} can be replaced by
\begin{equation}
\label{eq:index2}
\MUCZ(\bx)=n
\text{ and }
\MUCZ(\by)+n=\deg u .
\end{equation}
Furthermore, if $H$ is weakly non-degenerate, all inequalities in
\eqref{eq:index} are strict.
\end{Remark}

To obtain Theorem \ref{thm:gap2} as a consequence of Theorem
\ref{thm:gap21}, it suffices to notice that $I_\omega\mid_{H_*(M)}=0$.

\begin{proof}[Proof of Theorem \ref{thm:gap21}]
The first step in the argument is the following simple observation:

\begin{Lemma}
\label{lemma:gap2}
  Assume that quantum homology classes $u\in \HQ_*^-(M)$ and $w\in
  \HQ_*^-(M)$ and $\alpha\in\Lambda$ meet the homological condition of
  Theorems \ref{thm:gap2} and \ref{thm:gap21}. Let $H$ be a
  Hamiltonian on $M$ with isolated one-periodic orbits. Then
\begin{equation}
\label{eq:sel1}
-I_\omega(u)<\s_{[M]}(H)-\s_u(H)<I_\omega(\alpha)+I_\omega(w).
\end{equation}
\end{Lemma}

\begin{proof} Since $u=u*[M]$, where $u\in \HQ_*^-(M)$, and since the
  periodic orbits of $H$ are isolated, we have, by Proposition
  \ref{prop:LSF},
$$
\s_u(H)<\s_{[M]}(H)+I_\omega(u),
$$
which implies the first inequality in \eqref{eq:sel1}. Likewise, using
again Proposition \ref{prop:LSF}, we infer from the homological
condition that
$$
\s_{[M]}(H)<\s_{\alpha u}(H)+I_\omega(w)=\s_u(H)+I_\omega(\alpha)+I_\omega(w).
$$
Hence, the second inequality in \eqref{eq:sel1} also holds.
\end{proof}

Returning to the proof of the theorem, recall that the spectrality
condition, (AS10), holds due to the assumption that $M$ is rational
or, otherwise, $H$ is non-degenerate.  Hence, there exist capped
one-periodic orbits $\bx$ and $\by$ satisfying \eqref{eq:index} and
such that
$$
\CA_H(\bx)=\s_{[M]}(H)\text{ and } \CA_H(\by)=\s_u(H).
$$
Moreover, if $H$ is non-degenerate, we have \eqref{eq:index2}. Now the
action bound \eqref{eq:action} follows from \eqref{eq:sel1} and the
condition that $I_\omega(u)\geq 0$.

It remains to show that $x$ and $y$ are geometrically distinct, i.e.,
$\bx\neq \by$ and $\by$ is not a non-trivial recapping of $\bx$.  The
fact that $\bx\neq \by$ follows immediately from the first inequality
in \eqref{eq:action}: $\CA_H(\bx)-\CA_H(\by)>0$.

Moreover, if $\bx=\by\#A$, we necessarily have $I_\omega(A)>0$. Thus
$\CA_H(\bx)-\CA_H(\by)=\lambda_0k$, where $k$ is a positive integer,
when $M$ is rational.  In case (a), this is impossible since
$\CA_H(\bx)-\CA_H(\by)<\lambda_0$ by the second inequality of
\eqref{eq:action}.

Next, let us consider case (b), when $H$ is non-degenerate but $M$ is
not necessarily rational. Then \eqref{eq:index2} holds, and
\begin{equation}
\label{eq:degree}
0<|\MUCZ(\bx)-\MUCZ(\by)|<2N.
\end{equation}
Thus, if $\bx=\by\#A$, we have
$0<|I_{c_1}(A)|=|\MUCZ(\bx)-\MUCZ(\by)|<2N$, which is impossible.
\end{proof}

\begin{Remark}
  Note that in the last argument the non-degeneracy assumption on $H$
  is used twice: the first time to ensure the spectrality condition
  and the second time to prove, using \eqref{eq:degree}, that $\bx$ is
  not a recapping of $\by$. Hence, even if spectrality for degenerate
  Hamiltonian were established, the proof would still rely on the
  non-degeneracy assumption.
\end{Remark}

\begin{Remark}
  As is easy to see, Theorem \ref{thm:gap21} can be further
  generalized as follows. Let $M$ be closed and weakly monotone.
  Assume that there exists $v\in \HQ_*(M)$ and three classes $u$, $w$
  and $w'$ in $\HQ_*^-(M)$, and also $\alpha\in\Lambda$, such that
$$
u=v* w'\text{ and } v=(\alpha u)* w
$$
and $I_\omega(w')\leq 0$. Let the Hamiltonian $H$ have isolated
one-periodic orbits.  Assume furthermore that $M$ is rational and
$I_\omega(\alpha)+I_\omega(w)<\lambda_0$ or $0<|\deg v-\deg u|<2N$ and
$H$ is non-degenerate.  Then $H$ has two geometrically distinct
(capped) one-periodic orbits $\bx$ and $\by$ such that the inequality
\eqref{eq:action} holds with $u$ replaced by $w'$, and
\eqref{eq:index} holds with the first bound replaced by
$\big|\Delta_{H}(\bx)-\deg v \big|\leq 2n$. When $H$ is
non-degenerate, we have $\MUCZ(\bx)+n=\deg v$ and $\MUCZ(\by)+n=\deg
u$; cf.\ Remark \ref{rmk:gap21}. Setting $v=[M]$ and $w'=u$, we obtain
Theorem \ref{thm:gap21}.
\end{Remark}

\subsection{Proof of Theorem \ref{thm:CPn}}
By Example \ref{ex:cpn}, $\HQ_*(\CP^n)=\F$, when the degree $*$ is
even, and zero otherwise.  Let $u_{n-1}$ be a generator of
$\HQ_{2n-2}(\CP^n)$. Then $u_{n-i}:=u_{n-1}^i$ is a generator of
$\HQ_{2(n-i)}(\CP^n)$. In particular, $u_{n}=[M]$ and $u_0=[\pt]$ and
$u_{-1}=q^{-1}[M]$, etc., where $q=e^A$ with $I_{\omega}(A)=\lambda_0$
and $I_{c_1}(A)=2(n+1)$.

Consider first a Hamiltonian $H$ on $\CP^n$ with exactly $n+1$
distinct one-periodic orbits. Then, by Proposition \ref{prop:LSF},
there is a one-to-one correspondence between its capped one-periodic
orbits $\bx$ and the generators $u_k$ given by
$$
\CA_H(\bx)=\s_{u_i}(H).
$$
Set $\mu_H(\bx):=i$. For instance, if $\bx$ is non-degenerate,
$\mu_H(\bx)=\MUCZ(H,\bx)$.  Thus, the collection of capped orbits is
strictly ordered by the index $\mu_H$.  Furthermore, as readily
follows from Proposition \ref{prop:LSF} and Example \ref{ex:cpn}, this
collection is also strictly ordered by the action and these two
orderings coincide:
\begin{equation}
\label{eq:2orderings}
\mu_H(\bx)>\mu_H(\by) \text{ iff } \CA_H(\bx)>\CA_H(\by).
\end{equation}
This observation is in fact the key point of the proof. Note also
that, again by (AS10), we have
\begin{equation}
\label{eq:cpn-pf1}
n\leq |\Delta_H(\bx)-\mu_H(\bx)|\leq 2n.
\end{equation}

Assume now that $H$ is as in the statement of the theorem: for every
$k$, the flow of $H$ has exactly $n+1$ distinct $k$-periodic
orbits. (Then every such periodic orbit is necessarily the $k$th
iteration of a one-periodic orbit of $H$.)  The above considerations
apply to $H^{(k)}$ for all $k>0$. Recall also that
$$
\CA_{H^{(k)}}(\bx^k)=k\CA_H(\bx)\text{ and } 
\Delta_{H^{(k)}}(\bx^k)=k\Delta_H(\bx);
$$
by, e.g., (MI1). Hence, as follows from \eqref{eq:cpn-pf1},
\begin{equation}
\label{eq:cpn-pf2}
\lim_{k\to\infty}\frac{\mu_{H^{(k)}}\big(\bx^k\big)}{k}=\Delta_H(\bx).
\end{equation}
Fix two capped one-periodic orbits $\bx$ and $\by$ of $H$ with, say,
$$
\CA_H(\bx)>\CA_H(\by)\text{ and }\mu(\bx)>\mu(\by).
$$
Our goal is to show that $\tA_H(x)=\tA_H(y)$, where
$\tA_H=\CA_H-\lambda \Delta_H$. Without loss of generality, we may
assume that $\CA_H(\by)>0$.

Let $l$ be a positive integer such that 
$$
\CA_{H}\big(\by\#(lA)\big)=\CA_H(\by)+l\lambda_0>\CA_H(\bx).
$$
The sequence $\CA_{H^{(k)}}\big(\by^k\#
(lA)\big)=k\CA_H(\by)+l\lambda_0$ grows slower, as a function of $k$,
than the sequence $\CA_{H^{(k)}}\big(\bx^k\big)=k\CA_H(\bx)$.  Hence,
there exists a unique integer $k_l>1$ such that
$$
\CA_{H^{(k_l)}}\big(\by^{k_l}\# (lA)\big)\leq\CA_{H^{(k_l)}}\big(\bx^{k_l}\big),
\text{ but }
\CA_{H^{(k_l-1)}}\big(\by^{k_l-1}\# (lA)\big)> \CA_{H^{(k_l-1)}}\big(\bx^{k_l-1}\big).
$$

A straightforward calculation yields that
$$
\frac{k_l-1}{l}<\frac{\lambda_0}{\CA_H(\bx)-\CA_H(\by)}\leq \frac{k_l}{l},
$$
and hence
\begin{equation}
\label{eq:cpn-pf3}
\lim_{l\to\infty}\frac{k_l}{l}=\frac{\lambda_0}{\CA_H(\bx)-\CA_H(\by)}.
\end{equation}

Due to \eqref{eq:2orderings}, $k_l$ can be equivalently characterized by the
requirement
$$
\mu_{H^{(k_l)}}\big(\by^{k_l}\# (lA)\big)\leq \mu_{H^{(k_l)}}\big(\bx^{k_l}\big),
\text{ but }
\mu_{H^{(k_l-1)}}\big(\by^{k_l-1}\# (lA)\big)> \mu_{H^{(k_l-1)}}\big(\bx^{k_l-1}\big).
$$
Again, by a straightforward calculation, we have
$$
\frac{\mu_{H^{(k_l-1)}}\big(\by^{k_l-1}\big)}{l}+2(n+1)
> \frac{\mu_{H^{(k_l-1)}}\big(\bx^{k_l-1}\big)}{l}
$$
and
$$ 
\frac{\mu_{H^{(k_l)}}\big(\by^{k_l}\big)}{l}+2(n+1)
\leq \frac{\mu_{H^{(k_l)}}\big(\bx^{k_l}\big)}{l}.
$$
Letting $l\to\infty$ and taking into account \eqref{eq:cpn-pf2}
and \eqref{eq:cpn-pf3}, we obtain
$$
\frac{\Delta_H(\by)\lambda_0}{\CA_H(\bx)-\CA_H(\by)}+2(n+1)
=\frac{\Delta_H(\bx)\lambda_0}{\CA_H(\bx)-\CA_H(\by)}
$$
Since $\lambda_0=2(n+1)\lambda$, this readily implies the required equality
$$
\CA_H(\bx)-\lambda\Delta_H(\bx)=\CA_H(\by)-\lambda\Delta_H(\by),
$$
concluding the proof of the theorem.

\end{document}